\documentclass[11pt, reqno]{amsart}
\makeatletter
\usepackage{fullpage}
\usepackage{textcmds}  % amsrefs needs this, but it has to be loaded early to avoid re-defs.
\usepackage{amsmath, amssymb, amsfonts, amstext, verbatim, amsthm, mathrsfs}
\usepackage[mathcal]{eucal}
\usepackage{microtype}
\usepackage{graphicx}
\usepackage[all]{xy}

\usepackage[modulo]{lineno}
\usepackage[usenames]{color}
\usepackage{aliascnt} %To make autoref work as expected
\usepackage[colorlinks=true,linkcolor=blue,citecolor=blue,urlcolor=blue,citebordercolor={0 0 1},urlbordercolor={0 0 1},linkbordercolor={0 0 1}]{hyperref} %needs to be loaded after most things
\usepackage[shortalphabetic]{amsrefs} %needs to be loaded after hyperref
\usepackage{enumitem}
\usepackage{xspace}
\usepackage{tikz}
\usepackage{tabulary}

\theoremstyle{plain}

\newcommand{\refnewtheoremn}[4]{%
\newaliascnt{#1}{#2}
\newtheorem{#1}[#1]{#3}
\aliascntresetthe{#1}
\expandafter\providecommand\csname #1autorefname\endcsname{#4}}

\newcommand{\refnewtheorem}[3]{\refnewtheoremn{#1}{#2}{#3}{#3}}

\def\makeCal#1{%
\expandafter\newcommand\csname c#1\endcsname{\mathcal{#1}}}
\def\makeBB#1{%
\expandafter\newcommand\csname b#1\endcsname{\mathbb{#1}}}
\def\makeFrak#1{%
\expandafter\newcommand\csname f#1\endcsname{\mathfrak{#1}}}

\count@=0
\loop
\advance\count@ 1
\edef\y{\@Alph\count@}%
\expandafter\makeCal\y
\expandafter\makeBB\y
\expandafter\makeFrak\y
\ifnum\count@<26
\repeat
\newtheorem{thm}{Theorem}[section]

\refnewtheorem{prethm}{thm}{Pre-theorem}

\refnewtheorem{lem}{thm}{Lemma}
\refnewtheorem{claim}{thm}{Claim}
\refnewtheorem{cor}{thm}{Corollary}
\refnewtheorem{conj}{thm}{Conjecture}
\refnewtheorem{prop}{thm}{Proposition}
\refnewtheorem{quest}{thm}{Question}
\refnewtheorem{amplif}{thm}{Amplification}
\refnewtheorem{scholium}{thm}{Scholium}

\refnewtheorem{tablethm}{thm}{Table}
\refnewtheorem{assumption}{thm}{Assumption}
\refnewtheorem{conclusion}{thm}{Conclusion}
\refnewtheorem{problem}{thm}{Problem}

\theoremstyle{definition}
\refnewtheorem{rem}{thm}{Remark}
\refnewtheorem{defn}{thm}{Definition}
\refnewtheorem{notn}{thm}{Notation}
\refnewtheorem{const}{thm}{Construction}
\refnewtheorem{ex}{thm}{Example}
\refnewtheorem{fact}{thm}{Fact}

\newcommand{\op}[1]{\!\!\mathop{\rm ~#1}\nolimits}

\newcommand{\Coh}{\op{Coh}}
\newcommand{\stackCoh}{\mathcal{C}oh}
\newcommand{\DCoh}{\op{DCoh}}
\newcommand{\Spec}{\op{Spec}}
\newcommand{\iMap}{\mathcal{M}ap}

\newcommand{\ev}{\op{ev}}

\newcommand{\GL}{\op{GL}}
\newcommand{\Hom}{\op{Hom}}
\newcommand{\RHom}{\op{RHom}}
\newcommand{\Deg}{\op{Deg}}
\newcommand{\iDeg}{\mathcal{D}eg}
\newcommand{\pt}{\ast}
\newcommand{\iComp}{\mathcal{C}omp}
\newcommand{\Comp}{\op{Comp}}
\newcommand{\id}{\op{id}}
\newcommand{\Perf}{\op{Perf}}
\newcommand{\dgCat}{\op{dgCat}}
\newcommand{\Db}{\op{D}^b}
\newcommand{\res}{\op{res}}
\newcommand{\APerf}{\op{APerf}}
\newcommand{\QC}{\op{QCoh}}
\newcommand{\sod}[1]{\left\langle #1 \right\rangle}
\newcommand{\ladj}[1]{\beta^{< #1}}
\newcommand{\radj}[1]{\beta^{\geq #1}}
\newcommand{\Tstack}[2][]{\iMap(\Theta^{#1}, #2)}
\newcommand{\dual}{\vee}
\newcommand{\rank}{\op{rank}}
\newcommand{\gr}{\op{gr}}
\renewcommand{\ss}{\rm{ss}}
\newcommand{\us}{\rm{us}}

\makeatletter
\def\namedlabel#1#2{\begingroup
    #2%
    \def\@currentlabel{#2}%
    \phantomsection\label{#1}\endgroup
}
\makeatother

\begin{document}

\title{$\Theta$-stratifications, $\Theta$-reductive stacks, and applications}
\author{Daniel Halpern-Leistner}

\begin{abstract}
These are expanded notes on a lecture of the same title at the 2015 AMS summer institute in algebraic geometry. We give an introduction and overview of the ``beyond geometric invariant theory" program for analyzing moduli problems in algebraic geometry. We discuss methods for analyzing stability in general moduli problems, focusing on the moduli of coherent sheaves on a smooth projective scheme as a toy example. We describe several applications: a general structure theorem for the derived category of coherent sheaves on an algebraic stack; some results on the topology of moduli stacks; and a ``virtual non-abelian localization formula" in K-theory. We also propose a generalization of toric geometry to arbitrary compactifications of homogeneous spaces for algebraic groups, and formulate a conjecture on the Hodge theory of algebraic-symplectic stacks.
\end{abstract}

\maketitle

The problem of classification is one of the central meta-problems of mathematics. We present an approach to studying moduli problems in algebraic geometry which is meant as a synthesis of several different lines of research in the subject. Among the theories which fit into our framework: 1) geometric invariant theory, which we regard as the ``classification" of orbits for the action of a reductive group on a projective-over-affine scheme; 2) the moduli theory of objects in an abelian category, such as the moduli of coherent sheaves on a projective variety and examples coming from Bridgeland stability conditions; 3) the moduli of polarized schemes and the theory of $K$-stability.

Ideally a moduli problem, described by an algebraic stack $\cX$, is representable by a quasi-projective scheme. Somewhat less ideally, but more realistically, one might be able to construct a map to a quasi-projective scheme $q : \cX \to X$ realizing $X$ as the good moduli space \cite{alper2013good} of $\cX$. Our focus will be on stacks which are far from admitting a good moduli space, or for which the good moduli space map $q$, if it exists, has very large fibers. The idea is to construct a special kind of stratification of $\cX$, called a \emph{$\Theta$-stratification}, in which the strata themselves have canonical modular interpretations. In practice each of these strata is closer to admitting a good moduli space.

Given an algebraic stack $\cX$, our program for analyzing $\cX$ and ``classifying" points of $\cX$ is the following:
\begin{enumerate}
\item find a $\Theta$-reductive enlargement $\cX \subset \cX'$ of your moduli problem (See \autoref{defn:theta_reductive}), 
\item identify cohomology classes $\ell \in H^2(\cX';\bQ)$ and $b\in H^4(\cX';\bQ)$ for which the theory of $\Theta$-stability defines a $\Theta$-stratification of $\cX'$ (See \S \ref{sect:HN}),
\item prove nice properties about the stratification, such as the boundedness of each stratum.
\end{enumerate}
We spend the first half of this paper (\S 1 \& \S 2) explaining what these terms mean, beginning with a detailed review of the toy model of coherent sheaves on a projective scheme. Along the way we discuss constructions and results which may be of independent interest, such as a proposed generalization of toric geometry which replaces fans in a vector space with certain collections of rational polyhedra in the spherical building of a reductive group $G$ (\S \ref{sect:generalized_toric}).

In second half of this paper we discuss applications of $\Theta$-stratifications. In (\S 3 \& \S 4) we discuss how to use derived categories to categorify Kirwan's surjectivity theorem for cohomology (See \autoref{thm:Kirwan_surjectivity}), and several variations on that theme. Specifically, we discuss how methods of derived algebraic geometry and the theory of $\Theta$-stratifications can be used to establish structure theorems (\autoref{thm:derived_Kirwan},\autoref{thm:quasi_smooth_case})  for derived categories of stacks with a $\Theta$-stratification, and we use this to prove a version of Kirwan surjectivity for Borel-Moore homology (\autoref{cor:BM_surjectivity}). As an application we show (\autoref{thm:poincare}) that the Poincare polynomial for the Borel-Moore homology of the stack of Gieseker semistable sheaves on a K3 surface is independent of the semistability condition (provided it is generic), which leads to \autoref{conj:purity} on the Hodge theory of (0-shifted) symplectic derived Artin stacks. Finally in \S 4 we discuss how the same theory of (derived) $\Theta$-stratifications can be used to establish a ``virtual non-abelian" localization formula in $K$-theory which generalizes other virtual localization formulas for torus actions in $K$-theory and cohomology.

\subsubsection{Background}

We work throughout over the field of complex numbers for simplicity, although many of the results stated hold over a more general base scheme. For us the phrase ``moduli problem" is synonymous with ``algebraic stack" in the sense of Artin, which is a sheaf of groupoids on the big \'{e}tale site of commutative $\bC$-algebras such that the diagonal morphism $\cX \to \cX \times \cX$ is representable by algebraic spaces and there is a smooth surjective morphism from a scheme $X \to \cX$. Most of the time the diagonal will be quasi-affine.

\subsubsection{Author's note} I thank the organizers of the 2015 AMS Summer Institute in Algebraic Geometry for a lively conference and for inviting me to contribute to these proceedings. I thank Daniel Pomerleano and Pablo Solis for many useful comments on the first version of this manuscript. Some of the research described here serves as summary and announcement of work in preparation. The work described here was partially supported by the NSF MSPRF award DMS-1303960.

\section{The non-abelian Harder-Narasimhan problem}

\subsection{Motivational example: the Harder-Narasimhan filtration}

Fix a projective scheme $X$ over $\bC$. Our goal shall be to \emph{classify} all coherent sheaves on $X$ and how they vary in families. The story summarized here serves as the template for the theory of $\Theta$-stratifications, which seeks to extend this picture to more general moduli problems.

Fix an ample invertible sheaf $\cO_X(1)$ of Neron-Severi class $H \in NS(X)_\bR$. Given a coherent sheaf $E$ on $X$, the Grothendieck-Riemann-Roch theorem implies that the Hilbert function this is a polynomial of degree $d = \dim(\op{Supp}(E))$ whose coefficients can be expressed explicitly in terms of the Chern classes of $E$ and the class $H$, $P_E(n) = \chi(X,E\otimes \cO_X(n)) = \sum_i a_i(E) n^i$. For a flat family of coherent sheaves over a scheme $S$, which by definition is a coherent sheaf on $X \times S$ which is flat over $S$, the Hilbert polynomial of the restriction $E_s$ to each fiber is locally constant on $S$. We consider the moduli functor $\stackCoh(X)_P$, which is a contravariant functor from schemes to groupoids defined by
\[
\stackCoh(X)_P (T) = \left\{ F \in \Coh(X \times T) \left| F \text{ is flat over } S, \text{ and } P_{E_s}(t) = P(t), \forall s \in S \right. \right\}
\]
Similarly we let $\stackCoh(X)_{\dim \leq d}$ be the stack of families of coherent sheaves whose support has dimension $\leq d$.

$\stackCoh(X)_P$ is an algebraic stack locally of finite type over $\bC$,\footnote{The script font $\stackCoh(X)$ will denote the stack, and $\Coh(X)$ will denote the abelian category of coherent sheaves on $X$.} and in fact it can locally be described as a quotient of an open subset of a quot-scheme by the action of a general linear group. $\stackCoh(X)_P$ is not representable. In fact it is not even \emph{bounded}, meaning there is no finite type $\bC$-scheme $Y$ with a surjection $Y \to \stackCoh(X)_P$, i.e. there is no finite type $Y$ parameterizing a flat family of coherent sheaves such that every isomorphism class in $\Coh(X)$ appears as some fiber.

\begin{defn} \label{defn:semistability}
We consider the following polynomial invariants of a coherent sheaf $E \in \Coh(X)_{\dim \leq d}$
\[
\op{rk}(E) := P_E = \sum_{i=0}^d a_i(E) n^i, \qquad \op{deg}(E) := \sum_{i=0}^d (d-i) a_i(E) n^i
\]
We also define the \emph{polynomial slope} $\nu(E) := \deg(E) / \op{rk}(E)$, which is a well defined rational function of $n$ of $E \neq 0$ because $\op{rk}(E) \neq 0$. A coherent sheaf $E \in \Coh(X)_{\dim \leq d}$ is $H$-semistable if $\nu(F) \leq \nu(E)$ for all proper subsheaves $F \subset E$, by which we mean $\nu(F)(n) \leq \nu(E)(n)$ for all $n\gg 0$.\footnote{This is a slight reformulation of the polynomial Bridgeland stability condition discussed in \cite{bayer2009polynomial}*{\S 2}, which is itself a reformulation of Rudakov's reformulation \cite{rudakov1997stability} of Simpson/Gieseker stability. This notion of semistability agrees with that of \cite{huybrechts2010geometry}.} $E$ is \emph{unstable} if it is not semistable.
\end{defn}

For a family of coherent sheaves parametrized by $S$, the set of points $s \in S$ for which $E_s$ is semistable is open, hence we can define an open substack $\stackCoh(X)_P^{H-\ss} \subset \stackCoh(X)_P$ parameterizing families of semistable sheaves.

\begin{thm} \cite{huybrechts2010geometry}*{Theorem 4.3.4}
For every integer valued polynomial $P \in \bQ[T]$ of degree $\leq 2$ and every ample class $H \in NS(X)_\bR$, the stack of $H$-semistable coherent sheaves on $X$ admits a projective good moduli space $\stackCoh(X)_P^{H-\ss} \to M(X)_P^{H-\ss}$.
\end{thm}

We regard the scheme $M(X)_P^{H-\ss}$ as a solution of the classification problem for $H$-semistable sheaves: $M(X)_P^{H-\ss}$ does not quite represent the moduli problem, but the fibers of the map $\stackCoh(X)_P^{H-\ss} \to M(X)_P^{H-\ss}$ can be described fairly explicitly as ``$S$-equivalence" classes of semistable sheaves. This is not a complete classification of points of $\stackCoh(X)_P$, however, because we have discarded a huge closed substack of unstable sheaves. For the purposes of this paper, we are interested in the rest of the classification of coherent sheaves on $X$, and the structure of the unstable locus.

\begin{thm} \label{thm:HN}
If $E \in \Coh(X)_{\dim \leq d}$ is $H$-unstable, then there is a \emph{unique} filtration $E_N \subset E_{N-1} \subset \cdots \subset E_0 = E$ called the \emph{Harder-Narasimhan (HN)} filtration such that 
\begin{enumerate}
\item $\gr_i(E_\bullet) := E_i / E_{i+1}$ is semistable for all $i$, and
\item $\nu(\gr_0(E_\bullet)) < \nu(\gr_1(E_\bullet)) < \cdots < \nu(\gr_N(E_\bullet))$ for $n \gg0$.
\end{enumerate}
We refer to the tuple of Hilbert polynomials $\alpha = (P_{\gr^{HN}_0(E)},\ldots,P_{\gr^{HN}_N(E)})$ as the \emph{Harder-Narasimhan (HN) type} of $E$.

The subfunctor of $\cS_\alpha \subset \stackCoh(X)_P$ parametrizing families of unstable sheaves of HN type $\alpha = (P_0,\ldots,P_N)$ is a locally closed algebraic substack, and the assignment $[E] \mapsto (\gr^{HN}_0(E), \ldots,\gr^{HN}_N(E))$ defines an algebraic map
\[
\cS_\alpha \to \stackCoh(X)^{H-\ss}_{P_0} \times \cdots \times \stackCoh(X)^{H-\ss}_{P_N}.\\
\]
\end{thm}

\begin{rem}
Note that if $d' = \dim(\op{Supp}(E))$, then
\[
\nu(E) = d - d' + (d-d'+1) \frac{a_{d'-1}(E)}{a_{d'}(E)} \frac{1}{n} + O(\frac{1}{n^2}) \quad \text{as} \quad n \to \infty,
\]
and in particular $\lim_{n\to \infty} \nu(E) = d - d'$. This implies that a semistable sheaf must be pure (i.e. have no subsheaves supported on a subscheme of lower dimension), and in addition the sheaves $\gr^{HN}_i(E)$ are pure of dimension decreasing in $i$. Furthermore if $E \in \Coh(X)_{\dim \leq d}$ actually has support of dimension $d'<d$, then the slope $\nu_{\leq d'}(E)$ of $E$ regarded as an object of $\Coh(X)_{\dim \leq d'}$ is $\nu_{\leq d'}(E) = \nu_{\leq d}(E) - d+d'$. So the notion of semistability for $E \in \Coh(X)_{\dim \leq d'}$ agrees with that for $E$ regarded as an object of $\Coh(X)_{\dim \leq d}$.
\end{rem}

What this theorem means is that for any family of coherent sheaves $E$ over $S$, there is a finite algebraic stratification $S = \bigcup_\alpha S_\alpha$ by the HN type of the fiber $E_s$. For each stratum $S_\alpha$ the family $E|_{S_\alpha}$ is determined by a map $S_\alpha \to \stackCoh(X)^{H-\ss}_{P_0} \times \cdots \times \stackCoh(X)^{H-\ss}_{P_N}$ classifying the family $\gr^{HN}_\bullet(E_s)$, as well as some linear extension data encoding how $E|_{S_\alpha}$ can be reconstructed from $\gr^{HN}_\bullet(E|_{S_\alpha})$. We refer to the stratification $\stackCoh(X)_P = \stackCoh(X)_P^{H-\ss} \cup \bigcup \cS_\alpha$ as the Harder-Narasimhan-Shatz stratification. We regard \autoref{thm:HN} as a type of classification of coherent sheaves on $X$, as every coherent sheaf has been related in a controlled way to a point on some quasi-projective scheme.

\begin{rem}
Another important property of the Harder-Narasimhan filtration which we attempt to capture below is the fact that there is a total ordering on HN types such that the closure of a stratum $\cS_\alpha$ is contained in $\bigcup_{\beta \geq \alpha} \cS_\beta$.
\end{rem}

\subsubsection{Canonical weights for the Harder-Narasimhan filtration}

A $\bQ$-weighted filtration of $E \in \Coh(X)$ is a filtration $0 \subsetneq E_N \subsetneq \cdots \subsetneq E_0 = E$ along with a choice of rational weights $w_0 < w_1 < \cdots < w_p$. The second property of the HN filtration of an unstable bundle $E$ in \autoref{thm:HN} suggests that we regard the HN filtration as a $\bQ$-weighted filtration by choosing some $n\gg0$ and assigning weight $w_i = \nu(\gr^{HN}_i(E))(n)$. For reasons which will be clear below, we will only be interested in filtrations up to simultaneously rescaling the weights $(E_\bullet,w_\bullet) \mapsto (E_\bullet, k w_\bullet)$ for some $k>0$, so it suffices (by clearing denominators) to consider only $\bZ$-weighted filtrations.

Let $E \in \Coh(X)_{\dim \leq d}$, and define $D = \deg(E) \in \bQ[n]$ and $R = \op{rk}(E) \in \bQ[n]$. For any $\bZ$-weighted filtration $(E_\bullet,w_\bullet)$, define the numerical invariant
\begin{equation}\label{eqn:numerical_invariant}
\mu(E_\bullet,w_\bullet) = \lim_{n \to \infty} \frac{\sum_i \left(\deg(\gr_i(E_\bullet)) R - \op{rk}(\gr_i(E_\bullet)) D \right) w_i}{ \sqrt{\sum w_i^2 \op{rk}(\gr_i(E_\bullet))}} \in \bR \cup \{\pm \infty\}.
\end{equation}
We shall refer to a notion of semistability called \emph{slope semistability}, referred to as $\mu$-stability in \cite{huybrechts2010geometry}, which is coarser than that of \autoref{defn:semistability}.\footnote{It should be possible to modify the theory slightly to obtain the full actual Harder-Narasimhan filtration, but we will not pursue this here.} An analog of \autoref{thm:HN} holds -- the only difference being the stronger requirement that $\nu(\gr_i(E)) < \nu(\gr_{i+1}(E))$ must hold to leading order in $1/n$ -- and the HN filtration with respect to slope stability is obtained by deleting some terms of the HN filtration of \autoref{thm:HN}.

\begin{thm}\cites{halpern2014structure,zamora2014git,gomez2015git} \label{thm:numerical_invariant}
Let $E \in \Coh(X)$ be pure of dimension $d$. Then among all $\bZ$-weighted filtrations $(E_\bullet,w_\bullet)$ of $E$, there is a unique (up to rescaling weights) one which maximizes the numerical invariant $\mu(E_\bullet,w_\bullet)$. The filtration $E_\bullet$ is the Harder-Narasimhan filtration of $E$ with respect to slope semistability, and $w_i \propto a_{d-1}(\gr_i(E_\bullet)) / a_d(\gr_i(E_\bullet))$, which is the leading coefficient in $\nu(\gr_i(E_\bullet))$ as $n \to \infty$
\end{thm}

What is remarkable about \autoref{thm:numerical_invariant} is that it admits a formulation which makes no reference to the structure of the abelian category $\Coh(X)$ but only to the geometry of the stack $\cX = \stackCoh(X)$. Thus one obtains a framework, the theory of $\Theta$-stability and $\Theta$-stratifications, for generalizing this classification to other examples of moduli problems.

\subsection{Numerical invariants and the non-abelian Harder-Narasimhan problem}
\label{sect:HN}

We now formulate a notion of $\Theta$-stability and non-abelian HN filtrations which generalizes the previous discussion. We focus on a special set of ``test stacks," the most important of which is $\Theta := \bA^1 / \bG_m$, and we develop our theory of stability by considering maps out of $\Theta$.\footnote{Jochen Heinloth has also independently considered the general notion of semi-stability in terms of maps out of $\Theta$ that we present here, and he has introduced a beautiful method for showing that the semi-stable locus is separated in certain situations \cite{heinloth2016hilbert}. We will focus primarily on the \emph{unstable} locus here, although we hope in the future to connect the two stories.}

\begin{ex}
Our motivation for considering maps out of $\Theta$ is that the groupoid of maps $f : \Theta \to \stackCoh(X)_P$ is equivalent to the groupoid whose objects consist of a coherent sheaf $f(1) = [E] \in \stackCoh(X)_P$ along with a $\bZ$-weighted filtration of $E$. Maps $\pt / \bG_m \to \stackCoh(X)_P$ classify $\bZ$-graded coherent sheaves, and the restriction of a map $f : \Theta \to \stackCoh(X)_P$ to $\{0\}/\bG_m$ classifies $\gr_\bullet (E)$.
\end{ex}

Thus for a general algebraic stack $\cX$ we regard a map $f : \Theta \to \cX$ as a ``non-abelian filtration" of the point $f(1) \in \cX$. The groupoid of maps $B\bG_m \to \cX$ consists of points $p \in \cX(\bC)$ along with a homomorphism of algebraic groups $\bG_m \to \op{Aut}(p)$, and isomorphisms are isomorphisms $p_1 \simeq p_2$ so that the induced map $\bG_m \to \op{Aut}(p_1) \to \op{Aut}(p_2)$ is the given homomorphism for $p_2$. Given a non-abelian filtration $f : \Theta \to \cX$ the associated graded is the restriction $f_0 : \{0\}/\bG_m \to \cX$.

\begin{ex} \label{ex:theta_stack_quotients}
If $\cX = X/G$ is a global quotient stack, then the groupoid of maps $f : \Theta \to \cX$ is equivalent to the groupoid whose objects are pairs $(\lambda,x)$ where $\lambda : \bG_m \to G$ is a one-parameter subgroup and $x \in X(\bC)$ is a point such that $\lim_{t \to 0} \lambda(t) \cdot x$ exists. Isomorphisms are generated by conjugation $(\lambda,x) \mapsto (g \lambda g^{-1}, gx)$ for $g \in G$ and $(\lambda,x) \mapsto (\lambda,px)$ for $p \in P_\lambda = \{ g \in G | \lim_t \lambda(t) g \lambda(t^{-1}) \text{ exists} \}$.
\end{ex}

\begin{ex} \label{ex:polarized_schemes}
If $\cX$ is the moduli of flat families of projective schemes along with a relatively ample invertible sheaf, a non-abelian filtration $f : \Theta \to \cX$ classifies a $\bG_m$-equivariant family $X \to \bA^1$. There are the test-configurations studied in the theory of $K$-stability.
\end{ex}

The notion of $\Theta$-stability will depend on a choice of rational cohomology classes in $H^2(\cX;\bQ)$ and $H^4(\cX;\bQ)$, where by cohomology we mean the cohomology of the analytification of $\cX$ \cite{noohi2012homotopy}. Concretely if $\cX = X/G$, then $H^\ast(\cX;\bQ) \simeq H_G^\ast(X;\bQ)$. In particular one can compute $H^\ast(\Theta;\bQ) \simeq \bQ[q]$ with $q \in H^2(\Theta;\bQ)$ the Chern class of the invertible sheaf $\cO_\Theta\langle 1 \rangle$.

\begin{defn} \label{defn:numerical_invariant}
A \emph{numerical invariant} on a stack $\cX$ is a function which assigns a real number $\mu(f)$ to any non-degenerate map $f : \Theta \to \cX$ in the following way: given the data of
\begin{itemize}
\item a cohomology class $\ell \in H^2(\cX;\bQ)$, and
\item a class $b \in H^4(\cX;\bQ)$ which is \emph{positive definite} in the sense that for any $p \in \cX(\bC)$ and non-trivial homomorphism $\bG_m \to \op{Aut}(p)$, corresponding to a map $\lambda : B \bG_m \to \cX$, the class $\lambda^\ast b \in H^4(B\bG_m;\bQ) \simeq \bQ$ is positive,
\end{itemize}
the numerical invariant assigns $\mu(f) = f^\ast \ell / \sqrt{f^\ast b} \in \bR$.
\end{defn}

\begin{rem}
In \cite{halpern2014structure} we study a more general notion of a numerical invariant -- for instance we can allow $\mu$ to take values in a more general totally ordered set $\Gamma$ such as $\bR \cup \{\pm \infty\}$, or we can allow $b$ to be positive semi-definite -- but \autoref{defn:numerical_invariant} suffices for the purpose of exposition.
\end{rem}

\begin{lem}
There are classes $\ell \in H^2(\stackCoh(X)_P)$ and $b \in H^4(\stackCoh(X)_P)$ such that the numerical invariant of \autoref{defn:numerical_invariant} is given by the formula \eqref{eqn:numerical_invariant}.
\end{lem}

Note that the stack $\Theta$ has a ramified covering $z \mapsto z^n$ for every integer $n>0$. This scales $H^2(\Theta)$ by $n$ and scales $H^4(\Theta)$ by $n^2$, so a numerical invariant $\mu(f)$ is invariant under pre-composing $f : \Theta \to \cX$ with a ramified covering of $\Theta$. For instance, pre-composing a map $\Theta \to \stackCoh(X)_P$ with a ramified cover of $\Theta$ amounts to rescaling the weights of the corresponding weighted descending filtration. We can thus formulate

\begin{defn} \label{defn:nonabelian_HN}
Let $\mu$ be a numerical invariant on a stack $\cX$. A point $p \in \cX$ is called \emph{$\mu$-unstable} if there is a map $f : \Theta \to \cX$ with $f(1) \simeq p$ such that $\mu(f)>0$. A \emph{non-abelian Harder-Narasimhan (HN) filtration} of an unstable point $p$ is a map $f : \Theta \to \cX$ along with an isomorphism $f(1) \simeq p$ defined up to pre-composing $f$ with a ramified cover of $\Theta$.
\end{defn}

We refer to the question of existence and uniqueness of a non-abelian HN filtration for any unstable point in $\cX$ as the ``Harder-Narasimhan problem" associated to $\cX$ and $\mu$.

\begin{ex}
When $\cX = X/G$, \autoref{defn:nonabelian_HN} provides in intrinsic formulation of the Hilbert-Mumford criterion for instability in geometric invariant theory by letting $l = c_1(L)$ for some $G$-ample invertible sheaf $L$ and letting $b \in H^4(X/G;\bR)$ come from a positive definite invariant bilinear form $b \in \op{Sym}^2((\mathfrak{t}_\bR^\dual)^W) \simeq H^4(\pt/G;\bR)$.
\end{ex}

\begin{ex}
In the context of \autoref{ex:polarized_schemes}, one can find classes (See \cite{halpern2014structure}) in $H^2$ and $H^4$ such that $\mu$ is the normalized Futaki invariant of \cite{donaldson2005lower}. Thus $\Theta$-stability is a formulation of Donaldson's ``infinite dimensional GIT" in the context of algebraic geometry.
\end{ex}

\section{$\Theta$-reductive stacks}

Here we introduce a certain kind of moduli stack, which we refer to as a $\Theta$-reductive stack. These moduli stacks are natural candidates to admit $\Theta$-stratifications, as we shall discuss. By way of introduction, consider the main examples
\newline
\begin{center}
\begin{tabulary}{.9\textwidth}{C|C}
%\begin{tabular}{c|c}
\textbf{$\Theta$-reductive} & \textbf{Not $\Theta$-reductive} \vspace{2 pt}\\
\hline
\vspace{1 pt} $X / G$, where $G$ is reductive and $X$ is \emph{affine} &  \vspace{1 pt} $X/G$, where $G$ is reductive and $X$ is \emph{projective} \\
 \vspace{1 pt} $\stackCoh (X)$, where $X$ is a projective scheme &  \vspace{1 pt} the stack of vector bundles, or even the stack of torsion free sheaves on a proper scheme $X$ \\
 \vspace{1 pt} The stack classifying objects of some abelian category $\cA$ \footnote{The abelian category $\cA$ must satisfy suitable finiteness conditions in order for the moduli functor of classifying objects in $\cA$ to be an algebraic stack, but examples such as the category of modules over a finite dimensional algebra or the heart of a $t$-structure on $\DCoh(X)$ satisfying the ``generic flatness" condition will lead to algebraic moduli stacks.}  \vspace{1 pt} & \\
\end{tabulary}
%\end{tabular}
\end{center}
Note that in both cases a stack on the right hand side naturally admits an open immersion into a stack of the kind in the left column. If $X$ is a $G$-projective scheme with $G$-linearized very ample bundle $L$, then $X/G$ is an open substack of $\Spec (\bigoplus_{n \geq 0} \Gamma(X,L^n)) / \bG_m \times G$, the quotient of the affine cone over $X$. In fact, a close reading of the original development of geometric invariant theory \cite{mumford1994geometric} reveals that many statements in projective GIT are proved by immediately reducing to a statement on the affine cone. We shall refer to such an open embedding informally as an \emph{enlargement} of a moduli problem.

Recall that for two stacks $\cX,\cY$, one can always define the mapping stack
\[
\iMap(\cY,\cX) : T \mapsto \{ \text{maps } T \times \cY \to \cX\}.
\]
An old meta-theorem that when $\cY$ is proper, this mapping stack is actually algebraic. For instance, one can construct $\iMap(Y,X)$ explicitly using Hilbert schemes when $Y$ is a projective scheme and $X$ is quasi-projective. More generally, results of this kind have been established when $\cY$ is a proper scheme, algebraic space, or Deligne-Mumford stack \cite{olsson2006hom}, and when $\cX$ is a quasi-compact stack with affine diagonal \cite{DAGXIV}. In \cite{halpern2014mapping} we develop with Anatoly Preygel a theory of ``cohomological properness" for algebraic stacks, and we prove an algebraicity result for mapping stacks out of cohomologically proper stacks. In particular $\Theta$ is cohomologically proper, and we have:

\begin{thm}\cite{halpern2014mapping} \label{thm:Tstack}
Let $\cX$ be a locally finite type (derived or classical) algebraic stack with a quasi-affine diagonal. Then $\iMap(\Theta,\cX)$ is a locally finite type algebraic stack, and the evaluation map $\ev_1 : \iMap(\Theta,\cX) \to \cX$ which restricts a map to the open subset $\pt \simeq (\bA^1 -\{0\}) / \bG_m$ is relatively representable by locally finite type algebraic spaces.
\end{thm}

\begin{defn} \label{defn:theta_reductive}
Let $\cX$ be a locally finite type algebraic stack with a quasi-affine diagonal. Then $\cX$ is \emph{$\Theta$-reductive} if for any finite type ring $R$ and any $R$-point of $\cX$, the connected components of the fibers of the map $\ev_1 : \Tstack{\cX} \to \cX$ are proper over $\Spec(R)$.\footnote{We reserve the phrase \emph{weakly $\Theta$-reductive} for when the fibers of $\ev_1$ satisfy the valuative criterion for properness, without necessarily being quasi-compact.}
\end{defn}

\begin{ex}
One can show that $\Tstack{X/G} = \bigsqcup Y_\lambda / P_\lambda$, where the disjoint union is over conjugacy classes of one parameter subgroups $\lambda : \bG_m \to G$ and $Y_\lambda$ is the disjoint union of Bialynicki-Birula strata associated to $\lambda$ (compare \autoref{ex:theta_stack_quotients}). When $X$ is affine and $G$ is reductive, the map $\ev_1$ factors as the closed immersion $Y_\lambda / P_\lambda \hookrightarrow X / P_\lambda$ followed by the proper fibration $X / P_\lambda \to X/G$ with fiber $G/P_\lambda$. Therefore $\ev_1$ is proper on every connected component, so $X/G$ is $\Theta$-reductive in this case.
\end{ex}

\begin{ex}
Let $\Spec(R) \to \cX = \stackCoh(X)_P$ classify a $\Spec(R)$-flat family of coherent sheaves $F$ on $X \times \Spec(R)$. For any $R$-scheme $T$, the $T$-points of the algebraic space $Y = \Tstack{\cX} \times_{\cX} \Spec(R)$ classify $\bZ$-weighted filtrations of the coherent sheaf $F|_{X \times T}$ whose associated graded is flat over $T$. Thus the connected components of $Y$ can be identified with generalized flag schemes of $F$ over $\Spec(R)$, which are proper over $\Spec(R)$. As an exercise, we encourage the reader to use this to verify the claim that the moduli of vector bundles over a smooth curve is not $\Theta$-reductive.
\end{ex}

Given a point $p : \Spec(R) \to \cX$ we regard the connected components of the locally finite type algebraic space $\Tstack{\cX}_p := \Tstack{\cX} \times_{\ev_1,\cX,p} \Spec(R)$ as ``non-abelian flag varieties" for the moduli problem $\cX$, following the previous example.

\begin{ex}
When $\cX$ parametrizes objects in an abelian category $\cA \subset \DCoh(X)$ for a projective scheme $X$, one can show that the locally finite type algebraic spaces $\Tstack{\cX}_p$ satisfy the valuative criterion for properness, generalizing the example of $\stackCoh(X)_P$. This should hold for stacks classifying objects in more general abelian categories as well.
\end{ex}

%\begin{ex}
%Let $C$ be a smooth projective curve and let $\cX = \stackBun_{\GL_n}(C)$ be the stack of algebraic vector bundles on $C$. Because $\stackBun_{\GL_n}(C) \to \stackCoh(C)$ is an open immersion of stacks, it follows that $\Tstack{\stackBun_{\GL_n}(C)} \to \Tstack{\stackCoh(C)}$ is an open immersion as well \cite{}. Therefore the fibers of $\ev_1 : \Tstack{\stackBun_{\GL_n}(C)} \to \stackBun_{\GL_n}(C)$ of open subschemes of generalized flag schemes, and need not be proper. For instance, consider a line bundle $L$ and the injective map $\cO_C \to L \oplus L$ defined by two sections $s_1,s_2$. This two-step filtration of $L \oplus L$ defines a point in $\Tstack{\stackCoh(C)}_{[L\oplus L]}$ for any pair of sections $(s_1,s_2)$ taken up to simultaneous rescaling, however this point only lies in the open subscheme $\Tstack{\stackBun_{GL_2}(C)}_{[L\oplus L]}$ when the quotient of $L \oplus L$ by $\cO_C$ is locally free, which happens if and only if the divisors cut out by $s_1$ and $s_2$ have no points in common. We 
%thus see that $\stackBun_{\GL_2}(C)$ is not $\theta$-reductive.
%\end{ex}

Given a $\Theta$-reductive moduli problem, we can produce new $\Theta$-reductive moduli problems via the following:
\begin{lem}
If $\cX$ is a $\Theta$-reductive stack and $\cY \to \cX$ be a representable affine morphism, then $\cY$ is a $\Theta$-reductive stack.
\end{lem}

\begin{ex}
There are many natural moduli problems which are affine over the stack $\stackCoh(X)$ for a projective scheme $X$. For instance one can consider the moduli stack $\cY$ of flat families of coherent algebras on $X$, along with the map $\cY \to \stackCoh(X)$ which forgets the algebra structure. The fiber over a given $[F] \in \stackCoh(X)$ consists of an element of the vector space $\Hom_X(F \otimes F, F)$, corresponding to the multiplication rule, satisfying a finite set of polynomial equations, corresponding to the associativity and identity axioms. Other examples of stacks which are affine over $\stackCoh(X)$ include the stack of coherent modules over a fixed quasi-coherent sheaf of algebras on $X$. This includes as a special case the stack of (not necessarily semistable) Higgs bundles, which for smooth $X$ can be regarded as the stack of coherent sheaves of modules over the algebra $\op{Sym}_X (TX)$.
\end{ex}

%We also have the following
%\begin{lem}
%If $\cX \to Y$ is a map to an algebraic space, and $Y' \to Y$ is an fppf map such that $Y' \times_Y \cX$ is $\Theta$-reductive, then $\cX$ is $\Theta$-reductive.
%\end{lem}
%In particular, the recent result \cite{} implies that any stack which admits a good moduli space $\cX \to Y$ is \'etale locally over $Y$ a quotient of an affine scheme by a linearly reductive group, hence it is $\Theta$-reductive.

\subsubsection{The main advantage of $\Theta$-reductive stacks}

The primary importance of \autoref{defn:theta_reductive} is that the existence and uniqueness question for non-abelian Harder-Narasimhan filtrations is well behaved for points in a weakly $\Theta$-reductive stack. Below we will introduce a formal notion of numerical invariant $\mu$ on the stack $\cX$, as well as what it means for a numerical invariant to be \emph{bounded}. Before introducing the necessary machinery, however, let us state the main result that we are heading towards:

\begin{prop} \label{prop:HN}
Let $\cX$ be a stack which is (weakly) $\Theta$-reductive, and let $l \in H^2(\cX;\bQ)$ and $b \in H^4(\cX;\bQ)$ with $b$ positive definite. Assume that the numerical invariant $\mu(f) = f^\ast l / f^\ast b$ is bounded (\autoref{defn:bounded}). Then any unstable point $p \in \cX$ has a unique non-abelian Harder-Narasimhan filtration: i.e. there is a map $f : \Theta \to \cX$ with an isomorphism $f(1) \simeq p$ which maximizes $\mu(f)$, and this pair is unique up to ramified coverings $\Theta \to \Theta$.
\end{prop}

When $\cX = X/G$ is a global quotient stack, then boundedness holds automatically for any numerical invariant. So, if $X$ is affine and $G$ is reductive, this recovers Kempf's theorem on the existence of canonical destabilizing one parameter subgroups in GIT \cite{kempf1978instability}.

\subsection{The degeneration space: a generalization of the spherical building of a group}

In order to explain the boundedness hypothesis and the proof of \autoref{prop:HN}, we need the notion of the degeneration space -- a topological space associated to any point $p$ in an algebraic stack $\cX$. Let $\Theta^n = \bA^n / \bG_m^n$ and let $\mathbf{1} = (1,\ldots,1)$ be the generic point and $\mathbf{0} = (0,\ldots,0)$ be the origin. For any stack $\cX$ we call a map $f : \Theta^n \to \cX$ \emph{non-degenerate} if the map on automorphism groups $\bG_m^n \to \op{Aut} f(0)$ has finite kernel. For any $p \in \cX(\bC)$ consider the set\footnote{A priori this is a groupoid, but it turns out it is equivalent to a set.}
\[
\Deg(\cX,p)_n := \left\{ \text{nondegenerate maps } f: \Theta^n \to \cX \text{ with an isomorphism } f(1) \simeq p \right\}.
\]
We identify a full-rank $n \times k$ matrix $[a_{ij}]$ with nonnegative integer coefficients with a non-degenerate map $\phi : \Theta_k \to \Theta_n$ with an isomorphism $\phi(\mathbf{1})\simeq\mathbf{1}$ given in coordinates by $z_i \mapsto z_1^{a_{1i}} \cdots z_n^{a_{ni}}$, and this assignment is in fact a bijection between such matrices and maps of this kind. For any such map $\phi : \Theta^k \to \Theta^n$, pre-composition gives a restriction map $\Deg(\cX,p)_n \to \Deg(\cX,p)_k$, and this data fits together into a combinatorial object referred to as a \emph{formal fan} in \cite{halpern2014structure}. In addition we can associate an injective linear map $\bR_{\geq 0}^k \to \bR_{\geq 0}^n$ to such a $\phi = [a_{ij}]$, and this induces a closed embedding
$$\phi_\ast : \Delta^{k-1} = (\bR_{\geq 0}^k \setminus \{0\}) / \bR_{>0}^\times \hookrightarrow \Delta^{n-1} = (\bR_{\geq 0}^n \setminus \{0\}) / \bR_{>0}^\times$$

\begin{defn} \label{defn:degeneration_space}
We define the \emph{degeneration space} $\iDeg(\cX,p)$ of the point $p \in \fX(\bC)$ to be the union of simplices $\Delta^{n-1}_f$ for all $f \in \Deg(\cX,p)_n$, glued to each other along the closed embeddings $\phi_\ast : \Delta^{k-1}_{\phi^\ast f} \hookrightarrow \Delta^{n-1}_f$ for any $f \in \Deg(\cX,p)_n$ and any map $\phi : \Theta^k \to \Theta^n$ corresponding to a full-rank $n \times k$ matrix with nonnegative integer entries.
\end{defn}

This construction is quite similar to the geometric realization of a semi-simplicial set. In addition to being glued along faces, however, simplices can be glued to each other along any closed linear sub-simplex with rational vertices. This potentially leads to non-Hausdorff topologies, but that does not happen in the simplest examples:

\begin{ex} \label{ex:toric_degeneration}
If $X$ is a toric variety, $N$ is the co-character lattice of the torus $T$ acting on $X$, $\Sigma$ is the fan in $N_\bR$ describing $X$, and $p \in X$ is a generic point, then $\Deg(X/T,p)_n$ is the set of linearly independent $n$-tuples $(v_1,\ldots,v_n) \in N^n$ which are contained in some cone of $\Sigma$. It is a simple exercise to show that one can reconstruct the fan $\Sigma$ from the formal fan $\Deg(X/T,p)_\bullet \subset \Deg(\pt / T, \pt)_\bullet$, where the latter is identified with the set of linearly independent $n$-tuples $(v_1\ldots v_n) \in N^n$. The degeneration space $\iDeg(X/T,p)$ is homeomorphic to $(\op{Supp}(\Sigma) \setminus \{0\}) / \bR_{>0}^\times$, where $\op{Supp}$ denotes the support.
\end{ex}

\begin{ex} \label{ex:deg_space_group}
For a general group, $\Deg(\pt / G,\pt)_n$ classifies group homomorphisms with finite kernel $\phi : \bG_m^n \to G$ up to conjugation by an element of
\[
G_{\phi} = \{ g \in G | \lim_{t \to 0} \phi(t^{a_1},\ldots,t^{a_n}) g \phi(t^{a_1},\ldots,t^{a_n})^{-1} \text{ exists } \forall a_i \geq 0 \}.
\]
The correspondence assigns $\phi : \bG_m^n \to G$ to the composition $f_\phi : \Theta^n \to \pt / \bG_m^n \to \pt / G$. This in turn corresponds to a rational simplex $\Delta_\phi \to \iDeg(\pt / G,\pt)$, which is a priori just a continuous map, but one can show that it is a closed embedding in this case. When $G$ is semisimple, we can use simplices of this form to construct a homeomorphism $\iDeg(BG,\pt) \simeq \op{Sph}(G)$ \cite{halpern2014structure}*{Proposition 2.22} with the spherical building of $G$, i.e. the simplicial complex associated to the partially ordered set of parabolic subgroups of $G$ ordered under inclusion.
\end{ex}

%\begin{ex}
%Let $G$ be a semisimple group. The spherical building $\op{Sph}(G)$ is the simplicial complex associated to the partially ordered set of parabolic subgroups of $G$ ordered under inclusion, i.e. if $G$ has rank $r$ then there is an $(r-1)$-simplex associated to each Borel subgroup and two such simplices are glued together based on the parabolic subgroups they have in common.
%
%Consider a maximal torus $T$ with cocharacter lattice $N$, and choose an $r$-tuple $v = (v_1,\ldots,v_r) \in N^r$ which forms a set of ray generators for a Weyl chamber in $N_\bR$. This data describes a homomorphism with finite kernel $\phi_v : \bG_m^n \to G$ and hence a rational simplex $\Delta^{r-1}_{f_v} \hookrightarrow \iDeg(BG,\pt)$. On the other hand the choice of Weyl chamber determines a Borel subgroup $B \subset G$ and thus an $(r-1)$-simplex in $\op{Sph}(G)$. We can identify these two $(r-1)$-simplices, and in fact construct a homeomorphism $\iDeg(BG,\pt) \xrightarrow{\simeq} \op{Sph}(G)$ \cite{halpern2014structure}*{Proposition 2.22}.
%\end{ex}

\begin{ex} \label{ex:deg_space_general}
More generally, if $X$ is a scheme with a $G$ action and $p \in X$, then the canonical map $\iDeg(X/G,p) \to \iDeg(\pt/G,\pt)$ is a closed embedding. More explicitly, for any homomorphism with finite kernel $\phi : \bG_m^n \to G$ we associate a rational simplex $\Delta_\phi \hookrightarrow \iDeg(\pt / G,\pt)$. Let $\bG_m^n$ act diagonally on $\bA^n \times X$ via the homomorphism $\phi$ and the $G$ action on $X$, and let $X_{\phi,p}$ be the normalization of the $\bG_m^n$ orbit closure of $(1,\ldots,1,p) \in \bA^n \times X$. Then $X_{\phi,p}$ is a normal toric variety with a toric map to $\bA^n$, so the support of the fan $\Sigma_{X_{\phi,p}}$ of $X_{\phi,p}$ is canonically a union of rational polyhedral cones in $(\bR_{\geq 0})^n$. The closed subspace $\iDeg(X/G,p) \hookrightarrow \iDeg(\pt/G,\pt)$ is determined uniquely by the fact that
\[
\iDeg(X/G,p) \cap \Delta_{\phi} = (\op{Supp}(\Sigma_{X_{\phi,p}}) - \{0\}) / \bR_{>0}^\times \subset \Delta_{\phi}
\]
for each rational simplex $\Delta_\phi \subset \iDeg(\pt/G,\pt)$.
\end{ex}

\subsubsection{The space of components} Instead of considering the fiber of $\ev_1 : \Tstack[n]{\cX} \to \cX$ over a point as a set, for any map $\varphi : T \to \cX$ we can consider the fiber product $\Tstack[n]{\cX} \times_{\cX} T$ as an algebraic space over $T$.
\begin{defn}
Given $\varphi : T \to \cX$, we define the set $\Comp(\cX,\varphi)_n \subset \pi_0(\Tstack[n]{\cX} \times_{\cX} T)$ to consist of those connected components which contain at least one non-degenerate point. The \emph{component space} $\iComp(\cX,\varphi)$ is the union of standard $(n-1)$-simplices $\Delta_{[f]}^{n-1}$, one copy for each $[f] \in \Comp(\cX,\varphi)_n$, glued in the same manor as in the construction of $\iDeg(\cX,p)$ in \autoref{defn:degeneration_space}
\end{defn}

Any non-degenerate map $f : \Theta^n \to \cX$ with an isomorphism $f(1)\simeq p$ also defines an element $[f] \in \Comp(\cX,\id_{\cX})_n$, and the identity maps $\Delta_f^{n-1} \subset \iDeg(\cX,p) \to \Delta_{[f]}^{n-1} \to \iComp(\cX,\id_{\cX})$ glue for different $f \in \iDeg(\cX,p)_n$ to give a continuous map $\iDeg(\cX,p) \to \iComp(\cX,\id_{\cX})$. However, the component space tends to be much smaller that the degeneration space of any particular point.

\begin{ex} \label{ex:component_space_quotient}
When $\cX = X/G$ is a quotient of a $G$-quasi-projective scheme, then for a maximal torus $T \subset G$, the map $\Comp(\cX,\id_{\cX})_n \to \Comp(X/T,\id_{X/T})_n$ is surjective for all $n$, hence $\iComp(\cX,\id_{\cX}) \to \iComp(X/T,\id_{X/T})$ is surjective. One can show by explicit construction that $\iComp(X/T,\id_{X/T})$ can be covered by finitely many simplices, hence so can $\iComp(\cX,\id_{\cX})$.
\end{ex}

\subsection{A generalization of toric geometry?}
\label{sect:generalized_toric}

For a quotient stack $\cX = X/T$, where $X$ is a toric variety with generic point $p \in X$, we have seen in \autoref{ex:toric_degeneration} that $\Deg(\cX,p)_\bullet$ remembers the information of the fan in $N_\bR$ defining $X$, and hence remembers enough information to reconstruct $X$ itself. For a reductive group $G$, define a \emph{convex rational polytope} $\sigma \subset \iDeg(\pt / G, \pt)$ to be a closed subset such that for any homomorphism with finite kernel $\phi : \bG_m^n \to G$, the intersection of $\sigma$ with the rational simplex $\Delta_{\phi}$ as discussed in \autoref{ex:deg_space_group} is a convex polytope with rational vertices.

Now let $X$ be a normal projective-over-affine scheme with an action of a reductive group $G$ and a $p \in X$ such that $G \cdot p \subset X$ is dense. We have seen (\autoref{ex:deg_space_general}) that $\iDeg(X/G,p) \subset \iDeg(\pt/G,\pt)$ is a closed subspace whose intersection with any rational simplex $\Delta_\phi \subset \iDeg(\pt/G,\pt)$ is $(\op{Supp}(\Sigma_{X_{\phi,p}}) - \{0\}) / \bR_{>0}^\times$ for some toric variety $X_{\phi,p}$ with a toric map to $\bA^n$. The fan of $X_{\phi,p}$ therefore specifies a decomposition of $\Delta_\phi \cap \iDeg(X/G,p)$ into a union of rational polytopes. One can define a collection of convex rational polytopes $\Sigma_X = \{ \sigma \subset \iDeg(\pt/G,\pt) \}$ characterized by the properties that
\begin{enumerate}
\item the intersection of any two polytopes in $\Sigma_X$ is another polytope in $\Sigma_X$,
\item $\bigcup_{\sigma \in \Sigma_X} \sigma = \iDeg(X/G,p) \subset \iDeg(\pt/G,\pt)$, and
\item the intersections $\sigma \cap \Delta_\phi$ for $\sigma \in \Sigma_X$ are the polytopes in $\Delta_\phi \cap \iDeg(X/G,p)$ induced by the fan of $X_{\phi,p}$.
\end{enumerate}
\begin{quest}
To what extent does the resulting collection $\Sigma_X$ of rational polytopes in $\iDeg(\pt/G,\pt)$ remember the geometry of $X$?\footnote{It can not be that $\Sigma_X$ uniquely determines $X$: if $U \subset G$ is any unipotent subgroup group, then for $X = G/U$ one can check that $\iDeg(X/G,p) \simeq \iDeg(\pt / U,\pt) = \emptyset$ and hence $\Sigma_X = \emptyset$.}
\end{quest}

Let us formulate a more concrete question, assuming in addition that $X$ is smooth: The stabilizer group $\op{Stab}(p)$ acts on the space $\iDeg(X/G,p)$. We use the notation $C(Y)$ to denote the cone over a topological space $Y$, and observe that for a rational polytope $\sigma \subset \Delta_\phi$, the cone $C(\sigma)$ is canonically a rational polyhedral cone in $(\bR_{\geq 0})^n$, so we can consider polynomial functions on $C(\sigma)$. We define the ring of invariant piecewise polynomial functions on $\Sigma_X$,
\[
PP(\Sigma_X)^{\op{Stab}(p)} := \left\{ \begin{array}{c} \op{Stab}(p)\text{-invariant continuous } \\ f : C(\iDeg(X/G,p)) \to \bR \end{array} \left| \begin{array}{c} \forall \phi : \bG_m^n \to G, \forall \sigma \in \Sigma_X, \\ f|_{C(\Delta_\phi \cap \sigma)} \text{ is a polynomial} \end{array} \right. \right\} 
\]
In \cite{halpern2014structure}*{Lemma 2.27} we construct a homomorphism $H^{even}_G(X) \to PP(\Sigma_X)^{\op{Stab}(p)}$, where $H^{2n}$ maps to functions which are locally homogeneous polynomials of degree $n$.
\begin{quest}
Under what conditions is the map $H^{even}_G(X) \to PP(\Sigma_X)^{\op{Stab}(p)}$ an equivalence?
\end{quest}

For smooth toric varieties, this is known to be an equivalence, and it remains an equivalence for all toric varieties after replacing singular cohomology with a suitable alternative cohomology theory \cite{payne2006equivariant}. We have also verified that this map is an equivalence when $X = \pt/G$ itself, and when $X = G/P$ is a generalized flag manifold.

\subsection{The proof of \autoref{prop:HN}}

We can now finish explaining the terminology of \autoref{prop:HN} and its proof.

\begin{lem} \cite{halpern2014structure}*{Lemma 2.27}
Given a numerical invariant, the function $\mu = f^\ast \ell / \sqrt{f^\ast b}$ extends to a continuous function on $\iComp(\cX,\id_\cX)$ and thus a continuous function on $\iDeg(\cX,p)$ via the continuous map $\iDeg(\cX,p) \to \iComp(\cX,\id_{\cX})$ for any $p \in \cX$.
\end{lem}
\begin{proof}[Proof idea]
Any rational simplex $\Delta_{[f]}^{n-1}$ corresponds to \emph{some} non-degenerate map $f : \Theta^n \to \cX$. The classes $f^{\ast} \ell \in H^\ast(\Theta^n;\bQ)$ and $f^\ast b \in H^4(\Theta^n;\bQ)$ can be identified with a linear and positive definite quadratic form on $(\bR_{\geq 0})^n$ respectively.  The function $\mu$ restricted to $\Delta_{[f]}^{n-1}$ is simply the quotient $f^\ast \ell / \sqrt{f^\ast b}$, which descends to a continuous function on $(\bR^n_{\geq 0} - \{0\}) / \bR_{>0}^\times$.
\end{proof}

%\begin{rem}
%In fact, the general definition of a numerical invariant is an open subset $U \subset \iComp(\cX,\id_{\cX})$ and a continuous map $\mu :U \to \bR$ satisfying some mild conditions.
%\end{rem}

The proof of existence is essentially independent of the uniqueness proof. In light of the previous lemma, the existence of a maximizer for $\mu$ is an immediate consequence of the following property, which holds for \emph{any} numerical invariant on a global quotient stack by \autoref{ex:component_space_quotient}:

\begin{defn} \label{defn:bounded}
Let $\mu$ be a numerical invariant on $\cX$. We say that $\mu$ is \emph{bounded} if for all $p \in \cX(\bC)$ there is a finite collection of rational simplices in $\sigma_1,\ldots,\sigma_N \subset \iComp(\cX,p)$ such that for any point $x \in \iComp(\cX,p)$ there is a point $x' \in \sigma_i$ such that $\mu(x') \geq \mu(x)$.
\end{defn}

The uniqueness part of \autoref{prop:HN}, in contrast, uses the fact that $\cX$ is weakly $\Theta$-reductive in an essential way. Given an algebraic stack $\cX$ and $p \in \cX(\bC)$, we say that two rational points in $\iDeg(\cX,p)$ corresponding to $f,g \in \Deg(\cX,p)_1$ are \emph{antipodal} if there is a group homomorphism $\bG_m \to \op{Aut}(p)$ corresponding to a map $\lambda : B \bG_m \to \cX$, such that both $f,g : \Theta \to \cX$ can be factored as $f \simeq \lambda \circ \pi$ and $g \simeq \lambda^{-1} \circ \pi$, where $\pi : \Theta \to B\bG_m$ is the canonical projection.

\begin{lem} \label{lem:convexity}
Let $\cX$ be a (weakly) $\Theta$-reductive stack, and let $f,g \in \iDeg(\cX,p)$ be two distinct rational points which are not antipodal. Then there is a unique rational ray in $\iDeg(\cX,p)$ connecting $f$ and $g$.
\end{lem}

\begin{proof}[Proof idea]
%One can regard $\Theta$ as an open substack of $\bA^2 / \bG_m^2$ in two ways by rewriting $\Theta = \bA^1 \times (\bA^1 -\{0\}) / \bG_m^2$ and $\Theta = (\bA^1 - \{0\}) \times \bA^1 / \bG_m^2$. Given two maps $f,g : \Theta \to \cX$ with an isomorphism $f(1) \simeq g(1)$, one can glue them to a map $f \cup g : (\bA^2 - \{0\}) / \bG_m^2 \to \cX$. The statement of the lemma amounts to showing that this map can be extended uniquely to a map $\bA^2/\bG_m^2 \to \cX$ which is non-degenerate as long as $f$ and $g$ are non-degenerate and not antipodal.

%On the other hand, a map $\bA^2 / \bG_m^2 = \Theta \times \Theta \to \cX$ is by definition a map $\Theta \to \Tstack{\cX}$, and one can re-express the existence and uniqueness of the extension of $f \cup g$ to the existence and uniqueness of a dotted arrow making the following diagram commute:
A rational $1$-simplex is represented by a map $\bA^2 / \bG_m^2 = \Theta \times \Theta \to \cX$, which is equivalent to a map $\Theta \to \Tstack{\cX}$. One can reduce the claim to the existence and uniqueness of a dotted arrow making the following diagram commute:
\[
\xymatrix{ \Theta -\{0\} \simeq \pt \ar[r]^{f} \ar[d]^{1} & \Tstack{\cX} \ar[d]^{\ev_1} \\ \Theta \ar[r]^g \ar@{-->}[ur] & \cX },
\]
and such that the resulting map $\gamma : \bA^2/\bG_m^2 \to \cX$ is non-degenerate. The left vertical map is the open inclusion of the complement of a codimension $1$ closed point, so this lifting property essentially follows from the valuative criterion for the map $\ev_1$. The proof that $\gamma$ is non-degenerate is not classical: it uses the Tannakian formalism \cite{bhatt2015tannaka} and \autoref{thm:categorical_Kirwan} to show that if the homomorphism $\bG_m^2 \to \op{Aut}(\gamma(0,0))$ has a positive dimensional kernel then $f$ and $g$ must be either the same or antipodal.%, up to ramified covers, $\gamma$ factors through some map $\Theta^2 \to \Theta$ and that in fact in this case $f$ and $g$ must be either the same or antipodal.
\end{proof}

This lemma implies the uniqueness of a non-abelian Harder-Narasimhan filtration for $p \in \cX(\bC)$, assuming its existence. The sign of $\mu(f)$ must differ for antipodal points, so \autoref{lem:convexity} implies that for two rational points $f,g \in \iDeg(\cX,p)$ such that $\mu(f),\mu(g)>0$, there is a unique rational $1$-simplex connecting $f$ and $g$. Now we simply observe that the restriction of $\mu$ to this one simplex is the quotient of a linear form by the square root of a positive definite quadratic form. Such a function is strictly convex upward and therefore has a unique maximum on that interval.

\subsection{$\Theta$-stratifications}

The non-abelian Harder-Narasimhan problem is only the first step towards a classification theory for a stack $\cX$ modelled after that of $\stackCoh(X)_P$. Ideally the locus of $\Theta$-semistable points defines an open substack $\cX^{\ss} \subset \cX$, and the non-abelian Harder-Narasimhan filtration varies upper-semi-continuously in families. Fix a totally ordered set $\Gamma$.
\begin{defn} \label{defn:Theta_stratification}
A \emph{$\Theta$-stratum} in an algebraic stack $\cX$ is a closed substack which is identified with a connected component of $\Tstack{\cX}$ under the evaluation map $\ev_1 : \Tstack{\cX} \to \cX$. A \emph{$\Theta$-stratification} of $\cX$ is a collection of open substacks $\cX_{\leq c} \subset \cX$ for $c \in \Gamma$ such that $\cX_{\leq c'} \subset \cX_{\leq c}$ for $c'<c$ and the closed subset
\[
\cX_{\leq c} \setminus \bigcup_{c' <c} \cX_{\leq c'} = \bigsqcup \cS_\alpha
\]
is a disjoint union of $\Theta$-strata in $\cX_{\leq c}$.\footnote{Here $\alpha$ lies in some indexing set. The definition of $\Theta$-stratification in \cite{halpern2014structure} is slightly more involved, but it is equivalent to this condition for stacks $\cX$ which are open substacks of a $\Theta$-reductive stack by \cite{halpern2014structure}*{Lemma 3.7}.} The \emph{center} of $\cS_\alpha$ is defined to be the open substack $\cZ_\alpha^{\ss} \subset \iMap(\pt/\bG_m,\cX)$ which is the preimage of $\cS_\alpha \subset \Tstack{\cX}$ under the map $\sigma : \iMap(\pt/\bG_m,\cX) \to \Tstack{\cX}$ induced by the projection $\Theta \to \pt/\bG_m$. The map $\ev_0 : \cS_\alpha \subset \Tstack{\cX} \to \iMap(\pt/\bG_m, \cX)$ factors through $\cZ_\alpha^{\ss}$ as well.
\end{defn}

Consider a stack $\cX$ and a subset consisting of rational points $S \subset \iComp(\cX,\id_\cX)$ along with a map $\mu : S \to \Gamma$. We define the \emph{stability function} for $p \in \cX$:\footnote{The abuse of terminology here is justified, because typically $\Gamma = \bR$, and the map $\mu$ will be a numerical invariant, restricted to the subset $S$.}
\[
M^\mu(p) := \sup \left\{ \mu([f]) \left| f: \Theta \to \cX \text{ with } [f] \in S \text{ and } f(1) \simeq p \right. \right\}.
\]
Given this set-up, we shall denote by $\cS \subset \Tstack{\cX}$ the union of the connected components corresponding to $S$.

\begin{prop}
Let $\cX$ be a $\Theta$-reductive stack with $S \subset \iComp(\cX,\id_\cX)$ and $\mu : S \to \Gamma$ as above. Assume that
\begin{enumerate}
\item \textbf{HN-property:} For all $p \in \cX$, $\cS\times_{\cX} p$ is either empty or contains a unique point $f : \Theta \to \cX$ along with $f(1) \simeq p$, up to ramified covers of $\Theta$, such that $\mu([f]) = M^{\mu}(p)$, the \emph{non-abelian Harder-Narasimhan (HN) filtration}.
\item \textbf{Local finiteness:} For any map $\varphi : T \to \cX$, with $T$ a finite type scheme, there there is a finite subset $\{s_1,\ldots,s_n\} \subset S$ such that the non-abelian HN filtration of any $p \in T$ corresponds to a point in $s_i$ for some $i$.
\item \textbf{Semi-continuity:} If $f : \Theta \to \cX$ is a non-abelian HN filtration for $f(1)$, then $M^\mu(f(0)) \leq \mu([f])$.
\end{enumerate}
Then the subsets
\[
|\cX_{\leq c}| := \{p \in \cX | M^\mu(p) \leq c \} \subset |\cX|
\]
are open, and the corresponding open substacks $\cX_{\leq c}$ form a $\Theta$-stratification of $\cX$.
\end{prop}

One non-trivial consequence of the theory of numerical invariants is that the semi-continuity property follows automatically from the first two in the following special case: Given a numerical invariant $\mu : \iComp(\cX,\id_\cX) \to \bR$, we define the subset $S \subset \Comp(\cX,\id_\cX)_1$ to consist of those connected components of $\Tstack{\cX}$ which contain a point $f$ for which $\mu(f) = M^\mu(f(1))$. We restrict $\mu$ to define a function $S \to \Gamma = \bR$.

\begin{thm} \cite{halpern2014structure}*{Theorem 3.16}
Let $\mu$ be a numerical invariant on a $\Theta$-reductive stack. Assume that $\mu$ satisfies the HN property, and that for any finite type $T$ with a map $T \to \cX$, the non-abelian HN filtrations of points of $T$ correspond to finitely many points of $\iComp(\cX,\id_\cX)$. Then the subsets
\[
|\cX_{\leq c}| := \{p \in \cX | M^\mu(p) \leq c \} \subset |\cX|
\]
are open, and the corresponding open substacks of $\cX_{\leq c}$ form a $\Theta$-stratification of $\cX$.
\end{thm}

We use this in \cite{halpern2014structure} to construct a $\Theta$-stratification of the stack of objects an abelian subcategory $\cA \subset D^b(X)$ for a $K3$ surface $X$, corresponding to a Bridgeland stability condition on $D^b(X)$. This is the first example of a $\Theta$-stratification which is not known to admit a description as a certain limit of Kempf-Ness stratifications of larger and larger GIT problems as in \cites{zamora2014git, hoskins2014stratifications}.

\subsection{The search for $\Theta$-reductive stacks}

At the moment, the theory of $\Theta$-stability is still in a nascent stage, with much yet to be developed. One of the most intriguing problems is to find new examples of $\Theta$-reductive enlargements of classically studied moduli problems. For instance, we are not aware of a $\Theta$-reductive enlargement of the moduli of principal $G$-bundles on a smooth curve $C$, except in the case where $G = \GL_n$ and one has the $\Theta$-reductive enlargement $\stackCoh(X)$.

Another question, related to the theory of $K$-stability, is whether the moduli stack $\cX$ of flat families of projective schemes along with a relatively ample invertible sheaf is $\Theta$-reductive. \autoref{thm:Tstack} tells us that for abstract reasons there exists a locally finite type algebraic space $\Tstack{\cX}_{[(Y,L)]}$ classifying test configurations for a fixed polarized scheme $(Y,L) \in \cX$, and it is a foundational open question whether or not the connected components of this space are proper.

\section{Derived Kirwan surjectivity}

We now discuss applications of $\Theta$-stratifications. The original applications were for computing the topology of the stack $\cX$, in the case when $\cX$ is the moduli stack of vector bundles on a curve \cite{atiyah1983yang} and the stratification is the Harder-Narasimhan-Shatz stratification, or when $\cX = X/G$ is a quotient of a projective variety by a reductive group and the stratification is the Kempf-Ness stratification \cite{kirwan1984cohomology}. In our language, the theorem states
\begin{thm}[Kirwan surjectivity] \label{thm:Kirwan_surjectivity}
If $\cX$ is a \emph{smooth} locally finite type algebraic stack with a $\Theta$-stratification $\cX = \cX^{\ss} \cup \bigcup_\alpha \cS_\alpha$, then there is an isomorphism
\[
H^\ast(\cX; \bQ) \simeq H^\ast(\cX^{\ss};\bQ) \oplus \bigoplus_{\alpha} H^{\ast- c_\alpha} (\cZ_\alpha;\bQ)
\]
where $c_\alpha = \op{codim}(\cS_\alpha,\cX)$, assuming $c_\alpha \to \infty$ for large $\alpha$.
\end{thm}

The isomorphism in this theorem is not canonical: the theorem amounts to the fact that the canonical restriction map $H^\ast(\cX;\bQ) \to H^\ast(\cX_{\leq c};\bQ)$ is surjective for any $c$, and the direct sum decomposition depends on a choice of splitting for this surjection. 

\subsection{Categories and cohomology theories}

Our goal is to categorify Kirwan's surjectivity theorem, \autoref{thm:Kirwan_surjectivity} to say even more about the geometry of $\cX$. We adopt the perspective of non-commutative algebraic geometry. The assignment of a scheme $X$ to the dg-category of perfect complexes $\Perf(X)$ can be enhanced to a functor of $\infty$-categories from schemes to the $\infty$-category of dg-categories $\dgCat$, and we can regard this functor itself as a sort of enhanced cohomology theory. More precisely, one can define the category of non-commutative motives as a localization of the $\infty$-category of dg-categories (see \cite{blumberg2013universal} for the spectral case), and the functor $X \mapsto \Perf(X)$ is the assignment of a scheme to its non-commutative motive.

%The justification for regarding $X \to \Perf(X)$ as a categorification of cohomology is the classical result \cite{} that one can recover the cohomology of $X$ from the dg-category $\Perf(X)$ via ``periodic-cyclic homology." \footnote{In fact the periodic cyclic homology functor factors through the localization functor from dg-categories to noncommutative motives, i.e. $HP(\Perf(X))$ only depends on the $\Perf(X)$ regarded as a noncommutative motive.} We regard periodic cyclic homology as a functor $HP : \dgCat \to \op{Ch}$ from the $\infty$-category of dg-categories to the $\infty$-category of chain complexes.

%For dg-categories of equivariant perfect complexes, the periodic cyclic homology recovers the equivariant topological $K$-theory \cite{AS} of (the analytification of) $X$ with respect to a maximal compact subgroup $G_c \subset G$. More precisely, 

In \cite{blanc2013topological}, Anthony Blanc constructs a ``topological $K$-theory" spectrum $K^{top}(\cC)$ associated to any dg-category over $\bC$ and a natural Chern character map to the periodic cyclic homology $K^{top}(\cC) \otimes \bC \to HP(\cC)$ which is an equivalence when $\cC = \Perf(X)$ for a scheme $X$. Let $G_c \subset G$ be a maximal compact subgroup, and let $K_{G_c}(X)$ denote the equivariant topological $K$-theory \cites{segal1968equivariant,segal1970fredholm} of (the analytification of) $X$. With Daniel Pomerleano, we have shown that

\begin{thm}\cite{halpern2015equivariant}\label{thm:topological}
Let $X$ be smooth and projective-over-affine $G$-variety. Then there is a canonical equivalence $K^{top}(\Perf(X/G)) \xrightarrow{\simeq} K_{G_c}(X)$, and combined with the Chern character of \cite{blanc2013topological} this gives a canonical equivalence
\[
K_{G_c}(X) \otimes \bC \simeq K^{top}(\Perf(X/G)) \otimes \bC \xrightarrow{\simeq} HP(\Perf(X/G))
\]
\end{thm}

The method of proof actually uses \autoref{thm:categorical_Kirwan} to ``chop up" the noncommutative motive $\Perf(X/G)$ and realize it as a direct summand of a direct sum of countably infinitely many copies of the noncommutative motive of a quasi-projective Deligne-Mumford stack (or even a smooth quasi-projective scheme, after \cite{bergh2016geometricity}). This has many more applications, for instance constructing a functorial pure Hodge structure on the topological $K$-theory $K_{G_c}(X)$ when $X$ is smooth and projective-over-affine and $\Gamma(\cO_X)^G$ is finite dimensional via the noncommutative Hodge-de Rham sequence.\footnote{This intriguingly suggests the existence of a category of ``stacky motives" which remembers the topological $K$-theory of $\cX$ and also admits realization functors to the category of mixed Hodge structures.}

\subsubsection{Borel-Moore homology theories}

We will also be interested in the (co)homology theory which de-categorifies the assignment $\cX \to D^b(\cX)$, the derived category of coherent sheaves. When $\cX$ is not smooth, then $\Perf$ and $\Db$ do not agree, and in fact they have different functoriality properties. The assignment $\cX \mapsto \Db(\cX)$ is covariantly functorial for proper maps (because the derived pushforward of a bounded complex of coherent sheaves has bounded coherent cohomology) and contravariantly functorial for smooth maps $f : \cY \to \cX$ (or more generally flat maps), because those are the maps for which $f^\ast: D_{qc}(\cX) \to D_{qc}(\cY)$ maps $D^b(\cX)$ to $D^b(\cY)$. Furthermore the projection formula\footnote{Explain} guarantees that proper pushforward commutes with smooth pullback in a cartesian diagram.

Classically, the Borel-Moore homology of locally compact spaces has a similar kind of functoriality: covariantly functorial for proper maps and contravariantly functorial for open immersions. In fact, we can associate a Borel-Moore homology theory $E^{BM}$ to any cohomology theory (satisfying suitable axioms) on quotient stacks, regarded as a contravariant functor from the category of $G$-quasi-projective schemes and equivariant maps $E : (G\text{-quasi-proj})^{op} \to \cC$ to some stable $\infty$-category $\cC$, such as the category of spectra or chain complexes over a ring $R$, provided $E$ satisfies certain axioms. The defining properties of $E^{BM}$ are
\begin{itemize}
\item a version of Poincare duality $E^{BM}(X) = E(X)[2 \dim X]$ when $X$ is smooth, and 
\item a localization exact triangle $E^{BM}(Z) \to E^{BM}(X) \to E^{BM}(U)$ for any closed subscheme $Z \hookrightarrow X$ with complement $U = X \setminus Z$.
\end{itemize}
For any $G$-quasi-projective scheme we choose a $G$-equivariant embedding into a smooth $G$-quasi-projective scheme $Z \hookrightarrow X$ and define
\[
E^{BM}(Z) := \op{fib}(E(X) \to E(X\setminus Z))[2(\dim X-\dim G)] \in \op{Ho}(\cC).
\]
Provided the functor $E$ satisfies certain axioms (Explained in \cite{halpern2015equivariant}), $E^{BM}(Z)$ will be canonically independent, as an object of the homotopy category $\op{Ho}(\cC)$, of the choice of closed embedding of $Z \hookrightarrow X$. This definition of $E^{BM}$ is covariantly functorial for proper maps of $G$-quasi-projective schemes and contravariantly functorial for open immersions.

\begin{ex}
Let $E : (G\text{-quasi-proj})^{op} \to \op{Ch}$ corresponds to the usual theory of equivariant cohomology, for instance $E$ could assign a $G$-scheme to the complex of singular cochains on the homotopy quotient $EG_c \times_{G_c} X$, then $E^{BM}(Z)$ is equivariant Borel-Moore homology of $Z$ with integral coefficients, $H_i(E^{BM}(Z)) \simeq H^{BM}_{G_c,i}(Z)$.
\end{ex}

\begin{ex}
If $E(X) = K_{G_c}(X)$ is the topological $K$-theory spectrum (regarded as a naive $G$-spectrum) in the sense of \cite{segal1970fredholm}, then $K_{G_c}^{BM}(Z)$ is the Borel-Moore $K$-homology, sometimes referred to as ``K-homology with locally compact supports" \cite{thomason1988equivariant}. $K_{G_c,i}^{BM}(X)$ is a module over the Grothendieck ring of complex representations $\op{Rep}(G_c)$, as is $K^i_{G_c}(X)$.

The Atiyah-Segal completion theorem \cite{atiyah1969equivariant} states that the Chern character $\op{ch} : K^\ast_{G_c}(X;\bQ) \to H^\ast_{G_c}(X;\bQ)$ induces an isomorphism between the $2$-periodization $\prod_{n \in \bZ} H^{i+2n}_{G_c}(X;\bQ)$ and the completion $K^i_{G_c}(X;\bQ)^\wedge$ with respect to the \emph{augmentation ideal}, the kernel of the dimension homomorphism $\dim : \op{Rep}(G_c) \to \bZ$. It follows from the functoriality of $\op{ch}$ and the definition of $K^{BM}_{G_c}$ that there is a Chern character for the Borel-Moore theory inducing an equivalence
\[
\op{ch}_i : K^{BM}_{G_c,i}(X;\bQ)^{\wedge} \to \prod_{n\in \bZ} H^{BM}_{G_c,i+2n}(X;\bQ).
\]
So as in the case of ordinary equivariant $K$-theory, the equivariant Borel-Moore $K$-homology can be regarded as a ``spreading out" of equivariant Borel-Moore homology to a finitely generated module over $\op{Rep}(G_c)_\bQ$.
\end{ex}

\begin{thm}\cite{halpern2015equivariant} \label{thm:topological2}
For any $G$-quasi-projective scheme $X$, there is a canonical equivalence
\[
\rho_{G,X} : K^{top}(\Db(X/G)) \to K^{BM}_{G_c}(X)
\]
which commutes up to homotopy with proper pushforward, restriction to an open subset, and restriction of equivariance to a reductive subgroup $H \subset G$ such that $H$ is the complexification of $H \cap G_c$.
\end{thm}

\subsection{Categorical Kirwan surjectivity}

In light of the categorification of cohomology theories discussed above, our categorification of \autoref{thm:Kirwan_surjectivity} will amount to the following:
\begin{thm}\footnote{We refer the reader to the proof of \autoref{cor:BM_surjectivity} below for how to extract \autoref{thm:Kirwan_surjectivity} from \autoref{thm:splitting}.} \label{thm:splitting}
Let $\cX$ be a smooth locally finite type algebraic stack $\cX$ with a $\Theta$-stratification $\cX = \cX^{\ss} \cup \bigcup \cS_\alpha$. Then the dg-functor of restriction to the open substack $\res_{\cX^{\ss}} : \Perf(\cX) \to \Perf(\cX^{\ss})$ admits a right inverse $\op{ext} : \Perf(\cX^{\ss}) \to \Perf(\cX)$ such that $\res_{\cX^{\ss}} \circ \op{ext} \simeq \id_{\Perf(\cX^{\ss})}$.
\end{thm}

However, as stated \autoref{thm:Kirwan_surjectivity} discusses the structure of the unstable locus as well. In order to categorify the full statement of \autoref{thm:Kirwan_surjectivity}, we must recall the notion of a semiorthogonal decomposition. We say that a pre-triangulated dg-category $\cC$ admits a semiorthogonal decomposition $\cC = \langle \cC_i | i \in I \rangle$ indexed by a totally ordered set if the triangulated category $\op{Ho}(\cC)$ admits a semiorthogonal decomposition $\langle \op{Ho}(\cC_i) | i \in I \rangle$. By definition this means that $\RHom(E_i,E_j) = 0$ for $i > j$ and every object in $\op{Ho}(\cC)$ can be built as an iterated extension of objects of $\cC_i$ for $i$ in some finite subset of $I$. One of the basic results in the theory of semiorthogonal decompositions is that given a semiorthogonal decomposition $\cC = \langle \cC_i | i \in I \rangle$, every $F \in \cC$ can be built from a unique and functorial sequence of extensions.

\begin{ex}
A two term semiorthogonal decomposition $\cC = \langle \cA,\cB \rangle$ means that $\RHom(B,A) = 0$ for any $A \in \cA$ and $B \in \cB$, and for every $F \in \cC$ there is an exact triangle
\[
B \to F \to A \to
\]
with $B \in \cB$ and $A \in \cA$. This triangle is uniquely and functorially associated to $F$, and the assignment $F \mapsto A$ (respectively $F \mapsto B$) is the left (respectively right) adjoint of the inclusion $\cA \hookrightarrow \cC$ (respectively $\cB \hookrightarrow \cC$).
\end{ex}

Let us return to the context of a smooth stack $\cX$ with $\Theta$-stratification $\cX = \cX^{\ss} \cup \cS_0 \cup \cdots \cup \cS_N$. We shall assume that all of the stacks involved are finite type for simplicity. Note that every point of each $\cZ^{\ss}_\alpha$ has a canonical $\bG_m$ mapping to its stabilizer. We can thus decompose the category $\Perf(\cZ^{\ss}_\alpha) = \bigoplus_{w \in \bZ} \Perf(\cZ^{\ss}_\alpha)^w$, where $\Perf(\cZ^{\ss}_\alpha)^w$ is the full subcategory of $\Perf(\cZ^{\ss}_\alpha)$ consisting of complexes whose homology sheaves are all concentrated in weight $w$ with respect to the canonical $\bG_m$ stabilizer at each point. Given a complex $F \in \Perf(\cZ^{\ss}_\alpha)^w$, we say that $F$ has weights \emph{in the window} $[a,b)$ if it lies in the full subcategory $\bigoplus_{w=a}^{b-1} \Perf(\cZ^{\ss}_\alpha)^w$.

\begin{defn}
We define the full subcategories of $\Perf(\cX)$:
\begin{align*}
\cG^w &:= \left\{ F \in \Perf(\cX) | \forall \alpha, F|_{\cZ^{\ss}_{\alpha}} \text{ has weights in the window } [w_\alpha,w_\alpha + \eta_\alpha) \right\} \\
\Perf_{\cX^{\us}}(\cX)^{\geq w} &:= \left\{ F \in \Perf(\cX) | \op{Supp}(F) \subset \cX^{\us} \text{ and } \forall \alpha, F|_{\cZ^{\ss}_{\alpha}} \text{ has weights } \geq w_\alpha \right\} \\
\Perf_{\cX^{\us}}(\cX)^{< w} &:= \left\{ F \in \Perf(\cX) | \op{Supp}(F) \subset \cX^{\us} \text{ and } \forall \alpha, F|_{\cZ^{\ss}_{\alpha}} \text{ has weights } <w_\alpha + \eta_\alpha \right\}
\end{align*}
\end{defn}

The direct sum decomposition of \autoref{thm:Kirwan_surjectivity} can be categorified to a semiorthogonal decomposition of $\Perf(\cX)$ involving the categories above.

\begin{thm}\cites{halpern2015derived,ballard2012variation} \label{thm:categorical_Kirwan}
Let $\cX$ be a smooth finite type algebraic stack with a $\Theta$ stratification $\cX = \cX^{\ss} \cup \cS_0 \cup \cdots \cup \cS_N$. Then for any choice of weights $w = \{w_\alpha\}$ we have an infinite semiorthogonal decomposition
\[
\Perf(\cX) = \langle \Perf_{\cX^{\us}}(\cX)^{<w}, \cG^w, \Perf_{\cX^{\us}}(\cX)^{\geq w} \rangle.
\]
The restriction to $\cX^{\ss}$ defines an equivalence
\[
\res_{\cX^{\ss}} : \cG^w \xrightarrow{\simeq} \Perf(\cX^{\ss}),
\]
and $\Perf_{\cX^{\us}}(\cX)^{<w}$ (respectively $\Perf_{\cX^{\us}}(\cX)^{\geq w}$) further admits an infinite semiorthogonal decomposition whose pieces are identified with $\Perf(\cZ^{\ss}_\alpha)^{v}$ for all $\alpha$ and $v < w_\alpha$ (respectively $\Perf(\cZ^{\ss}_\alpha)^v$ for $v \geq w_\alpha$).
\end{thm}

The original formulation of this theorem in \cite{halpern2015derived} applies to quotient stacks, where the notion of a $\Theta$-stratification agrees with that of a ``KN-stratification." Some remarks on the case of local quotient stacks appear in \cite{halpern2015remarks}, and the final version is a consequence of the general theorems below. The main theorem of \cite{ballard2012variation} is not formulated as a structure theorem for the full equivariant derived category -- instead it analyzes a certain kind of \emph{elementary} (or \emph{balanced}) variation of GIT quotient and describes how the category $\cG^w$ (and thus the category $\Perf(\cX^{\ss})$) changes under this wall crossing. Matthew Ballard has extended this analysis to elementary wall crossings for smooth stacks which are not explicitly presented as global quotient stacks in \cite{ballard2014wall}. The notion of ``elementary stratum" introduced there is also a special case of a $\Theta$-stratum.

In this note, however, we focus on the generalization of \autoref{thm:categorical_Kirwan} to \emph{singular} stacks. Although \cite{halpern2015derived} contains some statements in the singular case, it became clear while finishing that paper that a much more general theorem holds, but expressing it most naturally requires the language of derived algebraic geometry.

\subsection{Derived Kirwan surjectivity}

In \autoref{thm:derived_Kirwan} below we shall extend \autoref{thm:categorical_Kirwan} to arbitrary stacks with a $\Theta$-stratification, but we must first establish some results and notation.

It is too much to ask for \autoref{thm:splitting} to hold for the category $\Perf(\cX)$ when $\cX$ is singular, because the restriction map $H^\ast(\cX;\bQ) \to H^\ast(\cX^{\ss};\bQ)$ is not always surjective.\footnote{One can consider the quotient of a cone over a projective variety by $\bG_m$.} Instead our general structure theorem applies to $\APerf(\cX) \subset \QC(\cX)$, the full subcategory of the unbounded derived category of quasi-coherent complexes $F$ such that the homology sheaves $H_i(F)$ are coherent and $=0$ for $i\ll 0$. When $\cX$ is ``quasi-smooth" there will be a version for $\DCoh(\cX)$ as well. The key idea will be to use the modular interpretation of a $\Theta$-stratum $\cS \hookrightarrow \cX$ as an open substack of the mapping stack $\Tstack{\cX}$ in order to equip $\cS$ with an alternative derived structure.

\subsubsection{The ``correct" derived structure on a $\Theta$-stratum}

First consider the example of an affine scheme $X = \Spec(R)$ with an action of $\bG_m$, corresponding to a grading on $R$. The $\Theta$-stratum associated to the tautological one parameter subgroup $\lambda(t) = t$ is $Y / \bG_m \hookrightarrow X/\bG_m$, where $Y = \Spec(R / R \cdot R_{>0})$ is the subscheme cut out by functions on $X$ with positive weights. The center of the stratum is $Z \times (\pt / \bG_m)$ where $Z = \Spec(R/ R \cdot (R_{>0}+R_{<0}))$. When $X$ is smooth we can identify the canonical short exact sequence
\begin{equation} \label{eqn:SES}
0 \to (N^\ast_Y X)|_Z \to (\Omega_X^1)|_Z \to (\Omega^1_Y)|_Z \to 0
\end{equation}
with the factorization of $(\Omega_X^1)|_Z$ into its weight eigensheaves $0 \to (\Omega_X^1)|^{>0}_Z \to (\Omega_X^1)|_Z \to (\Omega_X^1)|^{\leq 0}_Z \to 0$, and more generally this fact, that the restriction $(N^\ast_\cS \cX)|_{\cZ^{\ss}}$ has positive weights with respect to the canonical $\bG_m$ stabilizer of $\cZ^{\ss}$, plays a key role in the proof of both \autoref{thm:Kirwan_surjectivity} and its categorification \autoref{thm:categorical_Kirwan}.\footnote{We ask that the reader accept for the moment the importance of the positivity of these weights without justification. We will discuss the details of the results below, at which point it will be clear why this is important.}

%A version of this short exact sequence applies more generally to a $\Theta$-stratifications of a smooth stack $\cX$.\footnote{To modify \autoref{eqn:SES} for a $\Theta$-stratum in a general smooth stack, one must replace the sheaves $\Omega_X^1$ and $\Omega_Y^1$ with the two term complexes $L_{\cX}$ and $L_{\cS}$ respectively. For example, for a quotient stack $\cX = X/G$, the cotangent complex is the complex of equivariant vector bundles $L_\cX = [\Omega_X^1 \to \cO_X \otimes \mathfrak{g}^\ast]$, where $\Omega_X^1$ is in cohomological degree $0$ and the map is the infinitesimal action of the group $G$.} 

When $X$ is singular the sequence \eqref{eqn:SES} is no longer exact. The natural replacement is the exact triangle of derived restrictions of cotangent complexes
\begin{equation} \label{eqn:cotangent_sequence}
L_{\cS/\cX}|_{\cZ}[-1] \to L_{\cX}|_{\cZ} \to L_\cS|_{\cZ} \to,
\end{equation}
Unfortunately the homology sheaves of $(L_{\cS/\cX})|_{\cZ}$ can now fail to have positive weights.

\begin{ex}
Let $X = \Spec(\bC[x_1,\ldots,x_n,y_1,\ldots,y_m] / (f))$ where $f$ is a non-zero polynomial in $(x_1,\ldots,x_n)$ that is homogeneous of weight $d$ for the $\bG_m$ action on $\Spec(\bC[x_i,y_j])$ discussed above. Note that $d$ can be either positive or negative. Then because $f \in (x_1,\ldots,x_n)$ we have $Y = \Spec(\bC[x_i,y_j] / (f,x_i)) = \Spec (\bC[y_j])$, so the cotangent complex of $Y$ is $\cO_Y \cdot dy_1 \oplus \cdots \cO_Y \cdot dy_m$. The cotangent complex of $X$ is the two term complex in cohomological degrees $-1,0$ given by $\cO_X \cdot df \to \bigoplus_i \cO_X \cdot dx_i \oplus \bigoplus_j \cO_X \cdot dy_j$. Therefore applying the exact triangle \eqref{eqn:cotangent_sequence} to the inclusion $\cS = Y/\bG_m \hookrightarrow \cX = X/\bG_m$ we have
\[
L_{\cS/\cX} = \left[ \cO_Y \cdot df \to \bigoplus_i \cO_Y \cdot dx_i \right]
\]
in cohomological degrees $-2,-1$. Therefore, if $f$ had weight $d<0$ to begin with, the (derived) restriction of the relative cotangent complex $L_{\cS/\cX}|_{\cZ}$ no longer has weights $>0$.
\end{ex}

%The way to fix this example so that $L_{\cS/\cX}|_{\cZ}$ has positive weights is to consider an alternative derived structure on the stratum $Y$. Rather than defining $Y$ as $\Spec(R / R \cdot R_{>0})$, we replace $R$ with a quasi-isomorphic graded semi-free commutative differential graded algebra (CDGA) and \emph{then} quotient by the differential-graded ideal generated by elements of positive weight. In this case $\bC[x_i,y_j] / (f) \simeq \bC[x_i,y_j,\epsilon; d\epsilon = f(x_i,y_j)] =: A$ where epsilon is a free generator in homological degree $1$ and $\bG_m$-weight $d$. Then assuming $d<0$, $Y^{der}$ is the ``derived spec" of the CDGA $\cO_{Y^{der}} = A / (x_1,\ldots,x_n) \simeq \bC[y_j,\epsilon;d\epsilon = 0]$. Note that $H_0(\cO_{Y^{der}}) \simeq \cO_Y$, so this is a nilpotent ``derived thickening" of the stratum of the original example. Defining $\cS^{der} = Y^{der} / \bG_m$, one computes $L_{\cS^{der}/\cX} \simeq \bigoplus_i \cO_{\cS^{der}} \cdot dx_i [1]$, so the positivity of $L_{\cS/\cX}|_{\cZ}$ is restored by taking the correct \emph{derived structure} on the stratum $\cS$.

To fix this problem we shall equip $\cS$ with a different derived structure, such that $L_{\cS/\cX}|_{\cZ}$ has positive weights. If $\cX$ is a locally almost finitely presented algebraic derived stack with quasi-affine diagonal, then the \emph{derived} mapping stack $\Tstack{\cX}$ has a cotangent complex by \autoref{thm:Tstack}. If $\ev : \Theta \times \Tstack{\cX} \to \cX$ is the evaluation map and $p : \Theta \times \Tstack{\cX} \to \Tstack{\cX}$ is the projection, then the cotangent complex of the derived mapping stack is $L_{\Tstack{\cX}} \simeq (p_\ast \ev^\ast (L_\cX^\vee))^\vee$. This generalizes the classical observation that if $X$ and $Y$ are $\bC$-schemes with $X$ proper, then first order deformations of a map $f : X \to Y$ are classified by sections of $f^\ast TY$ on $X$, i.e. the tangent space of $\iMap(X,Y)$ is $\Gamma(X,f^\ast TY)$.

\begin{defn}
A \emph{derived} $\Theta$-stratum in a derived algebraic stack (locally almost finitely presented with quasi-affine diagonal) $\cX$ is a closed substack identified with a connected component of the \emph{derived} mapping stack under $\ev_1 : \Tstack{\cX} \to \cX$.
\end{defn}

Note that the underlying classical stack $\Tstack{\cX}^{cl}$ is the classical mapping stack to $\cX^{cl}$, so if $\cS \hookrightarrow \cX$ is a derived $\Theta$-stratum then $\cS^{cl} \hookrightarrow \cX^{cl}$ is a classical $\Theta$-stratum, and conversely if $\cS^{cl} \hookrightarrow \cX^{cl}$ is a classical $\Theta$-stratum, then it underlies a unique derived $\Theta$-stratum. So we can refer to a derived $\Theta$-stratum as a $\Theta$-stratum without ambiguity. One can use the formula $L_{\Tstack{\cX}} \simeq (p_\ast \ev^\ast (L_\cX^\vee))^\vee$ to prove the positivity of the weights of $L_{\cS/\cX}|_{\cZ}$.
\begin{lem} \label{lem:cotangent}
Let $\cS \hookrightarrow \cX$ be a derived $\Theta$-stratum, then there is a canonical equivalence of exact triangles in $\APerf(\cZ)$
\[
\xymatrix@R=7pt{ L_\cX|_{\cZ}^{>0} \ar[r]  \ar[d] & L_\cX|_{\cZ} \ar[r] \ar[d] & L_{\cX}|_\cZ^{\leq 0} \ar[r] \ar[d] & \\ L_{\cS/\cX}[-1]|_{\cZ} \ar[r] & L_{\cX}|_\cZ \ar[r] & L_\cS|_\cZ \ar[r] & }.
\]
\end{lem}

\subsubsection{Baric structures}

In addition to the strata $\cS$, we also equip the centers of the strata $\cZ$ with an alternate derived structure coming from the modular interpretation $\cZ \subset \iMap(B\bG_m,\cX)$. A mentioned previously, points of the stack $\cZ$ have a canonical $\bG_m$ in their automorphism groups, both in the classical and derived context. As in the classical case, any object of $F \in \QC(\cZ)$ splits canonically as a direct sum $\bigoplus_{w \in \bZ} F_w$, where the homology sheaves of $F_w$ locally have weight $w$ with respect to this canonical $\bG_m$-action.

\begin{defn}
Let $\cC$ be a pre-triangulated dg-category. Then a \emph{baric decomposition} on $\cC$ is a semiorthogonal decomposition $\cC = \sod{\cC^{<w},\cC^{\geq w}}$ for any $w \in \bZ$, with $\cC^{<w} \subset \cC^{<w+1}$ and $\cC^{\geq w} \subset \cC^{\geq w-1}$. Given a baric decomposition of $\cC$, we let $\radj{w}$ and $\ladj{w}$ denote the projection functors onto $\cC^{\geq w}$ and $\cC^{<w}$ respectively (they are also called the \emph{baric truncation functors}).
\end{defn}

The category $\QC(\cZ)$ has a baric decomposition which splits any $F \in \QC(\cZ)$ into a direct sum of complexes whose homology has weight $<w$ and $\geq w$ respectively, but we shall see that baric decompositions appear in more general settings. Observe that the stack $\Theta$ is a monoidal object in the category of stacks\footnote{In principle some care must be taken to ascribe $\Theta$ the structure of a monoidal object in the symmetric monoidal $\infty$-category of stacks, but we will only need to consider the monoidal structure maps up to homotopy, so we can get away with thinking of $\Theta$ as a monoidal object in the homotopy category of stacks.} where the monoidal product is given by the multiplication map $\bA^1 \times \bA^1 \to \bA^1$, which is equivariant for the group homomorphism $\bG_m \times \bG_m \to \bG_m$ which is also given in coordinates by $(z_1,z_2) \mapsto z_1 z_2$. Pre-composition $(t',f(t)) \mapsto f(t' \cdot t)$ defines an action of this monoidal object on the stack $\Tstack{\cX}$, giving an action map $a : \Theta \times \Tstack{\cX} \to \Tstack{\cX}$ in addition to the projection map $\pi : \Theta \times \Tstack{\cX} \to \Tstack{\cX}$.

\begin{lem} \label{lem:baric_decomp}
Let $\cS$ be a connected component of $\Tstack{\cX}$. Then $\QC(\cS)$ admits a baric decomposition where the baric truncation functors are defined by the exact triangle\footnote{Here $t$ is the canonical coordinate on $\Theta = \Spec (\bC[t]) / \bG_m$, which has weight $-1$, and $\cO_\Theta[t^{-1}]$ corresponds to the graded $\bC[t]$-module $\bC[t^\pm]$.}
\[
\xymatrix{ \pi_\ast (\cO_\Theta\langle w \rangle \otimes a^\ast(F)) \ar[r]^{t^w} \ar@{=}[d] & \pi_\ast(\cO_\Theta[t^{-1}] \otimes a^\ast(F)) \ar[r] \ar[d]^\simeq & \pi_\ast((\cO_\Theta[t^{-1}] / \cO_\Theta \cdot t^w) \otimes a^\ast(F)) \ar[r] \ar@{=}[d] & \\ \radj{w}(F) \ar[r] & F \ar[r] & \ladj{w}(F) \ar[r] & }.
\]
A complex $F \in \APerf(\cS)$ lies in $\QC(\cS)^{\geq w}$ (respectively $\QC(\cS)^{<w}$) if and only if the homology sheaves of $F|_{\cZ^{\ss}}$ are concentrated in weight $\geq w$ (respectively $<w$), where $\cZ^{\ss} \subset \iMap(B\bG_m, \cX)$ is the center of $\cS$.
\end{lem}

\begin{rem}
Note that the baric truncation functors preserves the subcategories $\Perf(\cS)$ and $\APerf(\cS)$, inducing baric decompositions of these categories as well.
\end{rem}

A version of this lemma is proved via concrete methods in \cite{halpern2015derived} for a classical $\Theta$-stratum in a classical quotient stack and in \cite{halpern2015remarks} for certain derived quotient stacks (using sheaves of CDGA's). In fact, \autoref{lem:baric_decomp} admits a purely formal proof which works for any stack $\cS$ with an action of the monoidal object $\Theta$, dramatically simplifying the previous proofs. Note that the monoidal structure on $\Theta$ equips the $\infty$-category $\QC(\Theta)$ with the structure of a co-monoidal object in the homotopy category of stable presentable $\infty$-categories. Then in fact any stable presentable $\infty$-category with a co-action of $\QC(\Theta)$ in the homotopy category of stable presentable $\infty$-categories, such as $\QC(\cS)$, admits a baric decomposition whose truncation functors are given by the formula of \autoref{lem:baric_decomp}. This will be explained in \cite{halpern2016derived}.

Another important baric structure arises in the following
\begin{lem} \label{lem:baric_decomp2}
Let $i : \cS \hookrightarrow \cX$ be a $\Theta$-stratum. Then there is a unique baric decomposition
\[
\APerf_\cS(\cX) = \sod{\APerf_\cS(\cX)^{<w},\APerf_\cS(\cX)^{\geq w}}
\]
such that $i_\ast : \APerf(\cS) \to \APerf_\cS(\cX)$ intertwines the baric truncation functors with the baric truncation functors on $\APerf(\cS)$ induced by \autoref{lem:baric_decomp}.\footnote{$\APerf_\cS(\cX)^{\geq w}$ consists of complexes in $\APerf_\cS(\cX)$ such that $i^\ast (F) \in \APerf(\cS)^{\geq w}$ and $\APerf_\cS(\cX)^{<w}$ consists of complexes which land in $\QC(\cS)^{<w}$ under the functor $i^{\QC,!} : \QC(\cX) \to \QC(\cS)$ right adjoint to $i_\ast$.} Furthermore, we can identify $\APerf_\cS(\cX)^w$ with the essential image of the fully faithful functor $i_\ast \pi^\ast : \APerf(\cZ)^w \to \APerf(\cX)$.
\end{lem}
This lemma is where one needs the positivity of the weights of relative cotangent complex -- more precisely, one needs $L_{\cS/\cX} \in \APerf(\cS)^{\geq 1}$, which holds for the derived $\Theta$-stratum by \autoref{lem:cotangent}.

\subsubsection{The main theorem}

Let $\cX = \cX^{\ss} \bigcup_\alpha \cS_\alpha$ be a derived stack with a $\Theta$-stratification, and denote the locally closed immersions $i_\alpha : \cS_\alpha \hookrightarrow \cX$. We choose $w_\alpha \in \bZ$ for each $\alpha$ in the indexing set of the stratification. We define
\begin{gather*}
\APerf(\cX)^{\geq w} := \left\{ F \in \APerf(\cX) \left| \forall \alpha, i_\alpha^\ast F \in \APerf(\cS_\alpha)^{\geq w_\alpha} \right. \right\} \\
\APerf(\cX)^{<w} := \left\{ F \in \APerf(\cX) \left| \forall \alpha, i_\alpha^{\QC,!} F \in \QC(\cS_\alpha)^{< w_\alpha} \right. \right\}.
\end{gather*}
Then we further define $\cG^w := \APerf(\cX)^{\geq w} \cap \APerf(\cX)^{<w}$, $\APerf_{\cX^{\us}}(\cX)^{\geq w} := \APerf_{\cX^{\us}}(\cX) \cap \APerf(\cX)^{\geq w}$, and $\APerf_{\cX^{\us}}(\cX)^{< w} := \APerf_{\cX^{\us}}(\cX) \cap \APerf(\cX)^{< w}$.
\begin{thm}\footnote{A proto-version of this theorem was established for derived global quotient stacks in \cite{halpern2015remarks}, following a more concrete approach using CDGA's to define alternate derived structures on the strata. Here we emphasize that the theory of $\Theta$-stratifications offers a much more natural approach, which is currently work in progress and will be discussed in detail in \cite{halpern2016derived}.}\label{thm:derived_Kirwan}
There are semiorthogonal decompositions
\begin{gather*}
\APerf_{\cX^{\us}}(\cX) = \sod{\APerf_{\cX^{\us}}(\cX)^{<w},\APerf_{\cX^{\us}}(\cX)^{\geq w}}, \text{ and} \\
\APerf(\cX) = \sod{\APerf_{\cX^{\us}}(\cX)^{<w}, \cG^w,\APerf_{\cX^{\us}}(\cX)^{\geq w}}.
\end{gather*}
Furthermore the restriction functor induces an equivalence $\cG^w \simeq \APerf(\cX^{\ss})$.
\end{thm}

This theorem is philosophically interesting for its lack of hypotheses: it shows that by working with the category $\APerf$ instead of the category $\Perf$ the statement of \autoref{thm:splitting} on the existence of a splitting of the restriction functor $\APerf(\cX) \to \APerf(\cX^{\ss})$ holds verbatim. On the other hand, the category $\APerf$ does not tend to de-categorify to any non-trivial invariants.\footnote{For instance, $K_0(\APerf(\cX)) = 0$ by a categorical version of Mazur's trick: For any $F \in \APerf(\cX)$, the complex $G := \bigoplus_{n\geq 0} F[2n]$ also lies in $\APerf$, and the identity $G = F \oplus G[2]$ implies that $[G] = [F] + [G]$ in $K$-theory and hence $[F]=0$.}

\autoref{thm:derived_Kirwan} has applications of its own, but the theorem's topological implications come from a modification which holds under more restrictive hypotheses.
\begin{defn}
A derived algebraic stack $\cX$ is quasi-smooth if $L_\cX$ is perfect and has tor-amplitude in $[-1,1]$.
\end{defn}

\begin{ex}
If $\cX_0,\cX_1 \to \cY$ are maps between smooth stacks, then the derived fiber product $\cX' := \cX_0 \times_{\cY} \cX_1$ is quasi-smooth. %, because of the exact triangle
%\[
%L_{\cX_0}|_{\cX'} \to  L_{\cX'} \to L_{\cX' / \cX_0} \simeq L_{\cX_1/\cY}|_{\cX'} \to.
%\]
%The first term is perfect with tor-amplitude in homological degree $[-1,0]$ and the second term is perfect with tor-amplitude in homological degree $[-1,1]$, which implies that $L_{\cX'}$ is perfect with tor-amplitude in homological degree $[-1,1]$.
%\end{ex}
%
%\begin{ex}
A special case is the ``derived zero locus" of a section $s$ of a locally free sheaf $V$ on a smooth stack $\cX$, which is by definition the derived intersection of $s$ with the zero section of $\op{Tot}_\cX(V)$. Under the projection to $\cX$ we can describe this more concretely as the relative derived $\Spec$ of the locally free sheaf of CDGA's $\op{Sym}_{\cO_\cX}(V^\dual[1])$ over $\cX$ with differential $V^\dual \to \cO_\cX$ given by the section $s$.
\end{ex}

\begin{ex}
If $S$ is a smooth surface, then $\cX = \stackCoh(S)$ is a quasi-smooth derived stack, because the fiber of the cotangent complex $L_\cX$ at a point $[E] \in \stackCoh(S)$ is
\[
L_{\cX,[E]} \simeq \RHom_S(E,E[1])^\dual \simeq \RHom_S(E,E)^\dual [-1].
\]
$\RHom$ between coherent sheaves has homology in positive cohomological degree only, so this combined with Serre duality $\RHom_S(E,E)^\dual \simeq \RHom(E,E\otimes \omega_S[2])$ implies that $L_{\cX,[E]}$ has homology in degree $-1,0,$ and $1$. Because $L_\cX \in \APerf(\stackCoh(S))$, this implies that $L_{\cX}$ is perfect with tor-amplitude in $[-1,1]$, hence $\stackCoh(S)$ is quasi-smooth.
\end{ex}

\begin{thm} \label{thm:quasi_smooth_case}
Let $\cX$ be a quasi-smooth derived algebraic stack with a $\Theta$-stratification $\cX = \cX^{\ss} \cup \bigcup_\alpha \cS_\alpha$, and assume that
\begin{itemize}
\item[\namedlabel{eqn:obstruction_weights}{$(\dagger)$}] for every point $f : \Spec(\bC) \to \cZ_\alpha^{\ss}$, the weights of $H_1(f^\ast L_{\cX})$ are $\geq 0$ with respect to the canonical $\bG_m$ acting on $f$.
\end{itemize}
Then for any choice of weights $\{w_\alpha\}$ the semiorthogonal decomposition of \autoref{thm:derived_Kirwan} induces a semiorthogonal decomposition
\[
\DCoh(\cX) = \sod{\DCoh_{\cX^{\us}}(\cX)^{<w}, \cG^w \cap \DCoh(\cX), \DCoh_{\cX^{\us}}(\cX)^{\geq w}},
\]
where the restriction functor induces an equivalence $\cG^w \cap \DCoh(\cX) \simeq \DCoh(\cX^{\ss})$. Furthermore, the baric decomposition of \autoref{lem:baric_decomp2} induces an infinite semiorthogonal decomposition of $\DCoh_{\cX^{\us}}(\cX)^{<w}$ (respectively $\DCoh_{\cX^{\us}}(\cX)^{\geq w}$) whose pieces are identified with $\DCoh(\cZ_\alpha^{\ss})^{v}$ for all $\alpha$ and $v<w_\alpha$ (respectively $v \geq w_\alpha$).
\end{thm}

If $\cX$ is a derived stack such that $L_\cX \simeq (L_\cX)^\dual$, then this implies that any $\Theta$-stratification satisfies the property \ref{eqn:obstruction_weights} and hence \autoref{thm:quasi_smooth_case} applies to $\cX$. In particular \autoref{thm:quasi_smooth_case} applies to any $\Theta$-stratification of a $0$-shifted derived algebraic symplectic stack in the sense of \cite{pantev2013shifted}. Simple examples of $0$-shifted derived symplectic algebraic stacks arise as the derived Marsden-Weinstein quotient of an algebraic symplectic variety by a hamiltonian action of a reductive group \cite{pecharich2012derived}. Another important class of $0$-shifted derived symplectic stacks are the moduli stacks $\cX = \stackCoh(S)_v^{H-\ss}$ for a $K3$ surface $S$ \cite{pantev2013shifted}. The equivalence $L_\cX = (L_\cX)^\dual$ comes from the Serre duality equivalence $\RHom_S(E,E[1]) \simeq \RHom_S(E[1],E[2])^\dual \simeq \RHom_S(E,E[1])^\dual$ for $E \in \Coh(S)$, or more precisely from a version of this equivalence in families.

%\begin{ex}
%Let $V$ be a symplectic linear representation of a reductive group $G$, i.e. a linear representation such that $G$ preserves a symplectic form $\omega \in \bigwedge^2(V^\dual)$. Then there is a canonical $G$-equivariant algebraic map $\mu : V \to \mathfrak{g}^\dual$ called the moment map. It is characterized by the fact that it is homogeneous of degree $2$ with respect to scaling, and $d\mu : \cO_V \otimes \mathfrak{g} \to \Omega_V^1 = \cO_V \otimes V^\dual$ is dual to the infinitesimal action map $\cO_V \otimes \mathfrak{g} \to TV = \cO_V \otimes V$ under the isomorphism $TV \simeq \Omega_V^1$ induced by $\omega$. Then the \emph{algebraic symplectic reduction} \cite{} is the quotient of the derived zero fiber $V_0 := \mu{-1}(0)$ by the group $G$. The cotangent complex of $\cX = V_0 / G$ is the self-dual complex $\cO_{V_0} \otimes \mathfrak{g} \to \cO_{V_0} \otimes V^\dual \to \cO_{V_0} \otimes \mathfrak{g}$. The GIT stratification of $\cX$ with respect to any linearization satisfies \ref{eqn:obstruction_weights}.
%\end{ex}

\section{Applications of derived Kirwan surjectivity}

Our first application of derived Kirwan surjectivity is an analog of \autoref{thm:Kirwan_surjectivity} for Borel-Moore homology and the Borel-Moore Poincare polynomial $P_t^{BM}(\cX) := \sum_{i\in \bZ} t^i \dim H^{BM}_i(\cX)$.

\begin{cor}\label{cor:BM_surjectivity}
Let $\cX$ be a quasi-smooth derived quotient stack with a $\Theta$-stratification satisfying \ref{eqn:obstruction_weights}. Then one has a direct sum decomposition
\[
H^{BM}_\ast(\cX;\bQ) \simeq H^{BM}_\ast(\cX^{\ss};\bQ) \oplus \bigoplus_\alpha H^{BM}_{\ast-2c_\alpha}(\cZ_\alpha^{\ss};\bQ)
\]
where $c_\alpha = \op{rank}(L_\cX|_{\cZ_\alpha^{\ss}}^{<0})$ and hence $P_t^{BM}(\cX) = P_t^{BM}(\cX^{\ss}) + \sum_\alpha t^{2 c_\alpha} P_t^{BM}(\cZ_\alpha^{\ss})$.
\end{cor}
\begin{proof}[Sketch of proof]
It suffices to verify the claim in the case of a single closed $\Theta$-stratum. The fact that the restriction functor $K^{BM}_{G_c,i}(X) \to K^{BM}_{G_c,i}(X^{\ss})$ is surjective for all $i$ follows immediately from \autoref{thm:quasi_smooth_case} and \autoref{thm:topological2}. The Atiyah-Segal completion theorem then implies that the restriction map $H^{BM}_{G_c,i}(X) \to H^{BM}_{G_c,i}(X^{\ss})$ is surjective for all $i$. It follows that the long exact ``localization" sequence in Borel-Moore homology becomes a short exact sequence for all $i$
\[
0 \to H^{BM}_i(S/G) \to H^{BM}_i(X/G) \to H^{BM}_i(X^{\ss}/G) \to 0,
\]
which split non-canonically.

So to prove the theorem it suffices to show that $H^{BM}_{\ast+c}(S/G) = H^{BM}_\ast(\cZ^{\ss})$ where $c = \op{rank}(L_\cX|_{\cZ^{\ss}}^{<0})$. The map of derived stack $\pi : \cS \to \cZ^{\ss}$ is itself relatively quasi-smooth of relative virtual dimension $c$, hence of finite tor-amplitude. A bit of work is required to show that $\pi$ induces a pullback map in Borel-Moore homology $\pi^\ast : H^{BM}_i(\cZ^{\ss}) \to H^{BM}_i(\cS)$ which is compatible with the pullback functor $\pi^\ast : \DCoh(\cZ^{\ss}) \to \DCoh(\cS)$ under the isomorphism of \autoref{thm:topological2} and the Atiyah-Segal completion theorem. Once that has been established, however, the infinite semiorthogonal decomposition of $\DCoh(\cS)$ of \autoref{thm:quasi_smooth_case} implies that $H^{BM}_{i+2c_\alpha} (\cS) \simeq H^{BM}_i(\cZ^{\ss})$ under this pullback map.
\end{proof}

\begin{rem}
These topological results are similar in spirit to those of \cite{kirwan1987rational}, where Kirwan proves that a version of \autoref{thm:Kirwan_surjectivity} holds for the intersection cohomology of singular $G$-varieties.
\end{rem}

\subsection{Topology of the moduli stack of sheaves on a $K3$ surface.}

Let $X$ be a smooth $K3$-surface. Then for Hilbert polynomials $P$ corresponding to ``primitive" numerical $K$-theory classes and for an ample class $H \in NS(X)_\bR$ avoiding a locally finite set of real codimension 1 ``walls", the stack $\stackCoh(X)_P^{H-\ss}$ is actually representable by a smooth hyperk\"{a}hler variety \cite{huybrechts2010geometry}*{\S4.C}.\footnote{Technically moduli functor as we have defined it is a $\bG_m$-gerbe over a hyperk\"ahler manifold, but that will not change any of the statements we make.} Assuming certain bounds on the numerical $K$-theory class, these varieties will be birational to each other, and so it follows that their Betti numbers agree \cite{batyrev1999birational}. In fact we show that if one uses Borel-Moore homology of the stack, then a statement of this form continues to hold.

\begin{thm} \label{thm:poincare}
For any two generic (i.e. avoiding an explicit locally finite collection of real hyperplanes) ample classes $H,H' \in NS(X)_\bR$ we have an equality of Borel-Moore Poincare polynomials $P_t^{BM}(\stackCoh(X)^{H-\ss}_P) = P_t^{BM}(\stackCoh(X)^{H'-\ss}_P)$.
\end{thm}

This is proved by analyzing the $\Theta$-stratification of $\stackCoh(X)_P$ as $H$ varies. In particular, say $H_+$ and $H_-$ lie on either side of a wall in $NS(X)_\bR$, and let $H_0 \in NS(X)_\bR$ be the point on the line segment joining $H_\pm$ which lies on this wall. Then $\stackCoh(X)_P^{H_\pm -\ss} \subset \stackCoh(X)_P^{H_0 -\ss}$, and the complement of this open substack is a union of Harder-Narasimhan strata with respect to $H_\pm$:
\begin{equation}\label{eqn:wall_crossing}
\stackCoh(X)_P^{H_- -\ss} \cup \bigcup_\alpha \cS_\alpha^{H_-} = \stackCoh(X)_P^{H_0 -\ss} = \stackCoh(X)_P^{H_+ -\ss} \cup \bigcup_\alpha \cS_\alpha^{H_+}
\end{equation}
It turns out that the set of Harder-Narasimhan types appearing in either $\Theta$-stratification are the same, so \autoref{cor:BM_surjectivity} implies an equality
\[
P_t(\stackCoh(X)_P^{H_- -\ss}) + \sum_\alpha t^{c_\alpha} P_t(\stackCoh(X)_\alpha^{H_- -\ss}) = P_t(\stackCoh(X)_P^{H_+ -\ss}) + \sum_\alpha t^{c_\alpha} P_t(\stackCoh(X)_\alpha^{H_+ -\ss}) ,
\]
where we are using the notation $\stackCoh(X)_\alpha^{H-\ss} = \stackCoh(X)_{P_0}^{H-\ss} \times \cdots \times \stackCoh(X)_{P_N}^{H-\ss}$ for a Harder-Narasimhan type $\alpha = (P_0,\ldots,P_N)$. This identity provides the basis for an inductive proof of the theorem.

\subsubsection{Speculations on the interaction with Hodge theory}

The Borel-Moore homology carries a mixed Hodge structure, and the fact that the restriction map of non-commutative motives ``with compact support" $[\DCoh(\cX)] \to [\DCoh(\cX^{\ss})]$ splits suggests that the restriction map $H^{BM}_i(\cX) \to H^{BM}_i(\cX^{\ss})$ splits as a map of mixed Hodge structures, so that the direct sum decomposition of \autoref{cor:BM_surjectivity} can be chosen compatibly with the Hodge structures. One might even suspect in this context that the restriction map of the motive with compact support $M^c(\cX) \to M^c(\cX^{\ss})$ in Voevodsky's big triangulated category of motives $DM(\bC,\bQ)$ (as constructed for quotient stacks in \cite{totaro2014motive}) splits. We have neither a proof, nor a counterexample.

On the other hand, one can imagine a different proof of \autoref{thm:poincare} along Hodge-theoretic lines. The virtual Hodge polynomial $E_{x,y}(\cX)$ of a stack \cite{joyce2007motivic} is additive for any stratification of $\cX$, so except for the comparison of $E_{x,y}(\cZ_\alpha^{\ss})$ with $E_{x,y}(\cS_\alpha)$, a version of \autoref{cor:BM_surjectivity} would hold automatically for the virtual Hodge polynomial. If the Hodge structure on $H^{BM}_i(\cX)$ is pure, then one can recover $P_t^{BM}(\cX)$ from $E_{x,y}(\cX)$, which leads us to
\begin{conj} \label{conj:purity}
The Hodge structure on $H_i^{BM}(\stackCoh(X)_P^{H-\ss})$ is pure for any $H$, and in general, $H_i^{BM}(\cX)$ is pure for any algebraic-symplectic (derived) stack $\cX$ which has a proper good moduli space. More ambitiously, one can ask if $H^i_{BM}(\cX)$ is pure for any algebraic-symplectic derived stack $\cX$ which is cohomologically proper in the sense of \cite{halpern2014mapping}.
\end{conj}
This conjecture is also inspired by the work of Ben Davison, who shows that the moduli of representations of the pre-projective algebra of any quiver has pure Borel-Moore homology in \cite{davison2016integrality}. In fact, he has an entirely different approach to proving results akin to \autoref{cor:BM_surjectivity} using the theory of mixed Hodge modules and purity theorems.

\subsection{Other applications of \autoref{thm:derived_Kirwan}.}

One of the largest open problems in the theory of derived categories of coherent sheaves is the conjecture that any two birational Calabi-Yau manifolds have equivalent derived categories of coherent sheaves, and there is a long history of results by Bondal, Orlov, Kawamata, Bridgeland, Bezrukavnikov, Kaledin and others constructing equivalences in larger and larger classes of examples. One of the main applications of \autoref{thm:categorical_Kirwan} is to establishing derived equivalences for flops between CY manifolds constructed explicitly via variation of GIT quotient. For instance, derived equivalences for the simplest kind of variation of GIT quotient were constructed in \cite{halpern2015derived} and \cite{ballard2012variation}. There has been a recent breakthrough in \cite{halpern2016combinatorial}, where derived equivalences for a much larger class of variation of GIT quotient (and hyperk\"ahler quotient) are established. The GIT problems studied in \cite{halpern2016combinatorial} serve as \'etale local models for the wall crossing \eqref{eqn:wall_crossing}, and in \cite{halpern2016derived} we will use these methods to prove that if $H,H\in NS(X)_\bR$ are two generic ample classes and $P$ is a primitive Hilbert polynomial, then $\DCoh(\stackCoh(X)_P^{H-\ss}) \simeq \DCoh(\stackCoh(X)_P^{H'-\ss})$.

\section{Non-abelian virtual localization theorem}

In addition to the applications we have discussed, one can use the notion of derived $\Theta$-stratifications to prove a ``virtual non-abelian localization" formula in $K$-theory.

\begin{defn}
A complex $F \in \Perf(\cX)$ is said to be \emph{integrable} if $\bigoplus_i H_i R\Gamma(\cX,F \otimes G)$ is finite dimensional for any $G \in \DCoh(\cX)$.
\end{defn}
The subcategory of integrable complexes is a stable idempotent complete $\otimes$-ideal of $\Perf(\cX)$, whose $K$-theory is a somewhat adhoc version of algebraic $K$-theory ``with compact supports."

\begin{ex}
If $\cX$ has a good moduli space $\cX \to Y$ such that $Y$ is proper -- for instance the map $X/G \to X/\!/G$ where $X$ is a $G$-quasi-projective-scheme which admits a projective good quotient $X/\!/G$ -- then any $F \in \Perf(\cX)$ is integrable. The prototypical example is the moduli stack of semistable principal $G$-bundles on a smooth curve.
\end{ex}

\begin{ex}
More generally if $F \in \Perf(\cX)$ is set theoretically supported on a closed substack $\cY \subset \cX$ which admits a proper good moduli space $\cY \to Y$ then $F$ is integrable. Even more generally, $F \in \Perf(\cX)$ is integrable whenever the support of $F$ is cohomologically proper in the sense of \cite{halpern2014mapping}.
\end{ex}

For an integrable complex $F \in \Perf(\cX)$,\footnote{Here if $\cX$ is derived, then we assume that $H_i(\cO_\cX) = 0$ for $i\gg0$ so that $\cO_\cX \in \DCoh$.} one can define the $K$-theoretic ``index" to be $\chi(\cX,F)$, which is the $K$-theoretic analog of the integral of a compactly supported cohomology class. In fact when $\cX = X$ is a scheme the two notions of integration are directly related via the Grothendieck-Riemann-Roch theorem. So we see that integration in $K$-theory, as opposed to integration in cohomology, generalizes easily from schemes to stacks.

When $\cX$ is quasi-smooth derived stack, the index $\chi(\cX,F)$ of an admissible $F \in \Perf(\cX)$ is analogous to the integral of a compactly supported cohomology class on $\cX^{cl}$ with respect to a \emph{virtual fundamental class}. To make this precise, consider the surjective closed immersion $\iota : \cX^{cl} \hookrightarrow \cX$. In $K_0(\DCoh(\cX))$ we have $[\cO_\cX] = \sum_i (-1)^i [H_i(\cO_\cX)]$. The objects $H_i(\cO_\cX)$ are canonically $\iota_\ast$ of coherent sheaves on $\cX^{cl}$, which we denote $H_i(\cO_\cX)$ as well, and we introduce the \emph{virtual structure sheaf} $\cO_{\cX}^{vir} =\bigoplus_i H_i(\cO_\cX)[i] \in \DCoh(\cX^{cl})$. By the projection formula we have
\[
\chi(\cX,F) = \chi(\cX, F \otimes \iota_\ast(\bigoplus H_i(\cO_\cX)[i])) = \chi(\cX^{cl}, \cO_{\cX}^{vir} \otimes \iota^\ast F).
\]
Virtual integrals of this form, along with their cohomological counterparts, play a central role in modern enumerative geometry.

Given a $\Theta$-stratification of a quasi-smooth derived stack $\cX = \cX^{\ss} \cup \bigcup_\alpha \cS_\alpha$, the virtual non-abelian localization formula relates the index of a $K$-theory class on $\cX$ to the index of its restriction to $\cX^{\ss}$ as well as $\cZ_\alpha^{\ss}$ for all $\alpha$. Let us define $L_\alpha^+ := \radj{1} (L_\cX|_{\cZ_\alpha^{\ss}}) \in \APerf(\cZ_\alpha^{\ss})^{\geq 1}$ and $L_\alpha^- := \ladj{0}(L_\cX|_{\cZ_\alpha^{\ss}}) \in \APerf(\cZ_\alpha^{\ss})^{<0}$. Because $\cX$ is quasi-smooth, $L_\cX$ is actually perfect, and hence so are $L_\alpha^+$ and $L_\alpha^-$. For each stratum, we define a complex\footnote{This complex $E_\alpha$ can be regarded as a $K$-theoretic reciprocal of the virtual normal bundles of $\cZ_\alpha$ in $\cX$.}
\[
E_\alpha = \op{Sym}(L_\alpha^- \oplus (L_\alpha^+)^\dual) \otimes (\det L_\alpha^+)^\dual [-\rank L_\alpha^+] \in \QC(\cZ_\alpha^{\ss}).
\]
Note that $L_\alpha^+[1]$ is $L_{\cS_\alpha / \cX}|_{\cZ_\alpha^{\ss}}$ once one equips $\cS_\alpha$ with the ``correct" derived structure discussed above. We also define the integer $\eta_\alpha$ to be the weight of $\det(L_\alpha^+)$.

\begin{defn} \label{defn:admissible}
We say that $F\in \Perf(\cX)$ is \emph{almost admissible} if $F|_{\cZ_\alpha^{\ss}} \in \Perf(\cZ_\alpha^{\ss})^{<\eta_\alpha}$ for all but finitely many $\alpha$.
\end{defn}

\begin{ex}
Consider the moduli stack $\cX = \op{Bun}_G(C)$ of principal $G$-bundles on a smooth curve $C$. Then $\cX$ is smooth, and the Shatz-Harder-Narasimhan stratification is a $\Theta$-stratification. In \cite{teleman2009index}, Teleman and Woodward define the subspace of ``admissible classes" in $K_0(\op{Bun}_G(C))$ as the span of some explicit complexes and prove formulas for their index, $\chi(\op{Bun}_G(C),F)$. These classes are almost admissible in the sense of \autoref{defn:admissible}.\footnote{The ``almost" in the terminology is to be consistent with \cite{halpern2016equivariant}. There it was convenient to define for an invertible sheaf $\cL$ on $\cX$ the class of $\cL$-admissible complexes to be those $F \in \Perf(\cX)$ for which $\cL \otimes F^{\otimes m}$ is almost admissible for all $m>0$. One can say that the Atiyah-Bott classes are $\cL$-admissible for any $\cL \in \op{Pic}(\op{Bun}_G(C))$ which has a positive level.}
\end{ex}

\begin{thm}[Virtual non-abelian localization] \label{thm:virtual_localization}
Let $\cX$ be a quasi-smooth derived stack with a $\Theta$-stratification $\cX = \cX^{\ss} \cup \bigcup_\alpha \cS_\alpha$. Let $F \in \Perf(\cX) \subset \DCoh(\cX)$ be an almost admissible complex such that $F|_{\cZ_\alpha}$ is integrable for all $\alpha$. Then $R\Gamma(\cX,F)$ is finite dimensional if and only if $R\Gamma(\cX^{\ss},F)$ is, and in this case we have
\[
\chi(\cX,F) = \chi(\cX^{\ss},F) + \sum_{\alpha} \chi(\cZ_\alpha^{\ss},F|_{\cZ_\alpha^{\ss}} \otimes E_\alpha).
\]
\end{thm}

\begin{ex}[Abelian virtual localization]
Let $X$ be a quasi-smooth scheme with an action of a torus $T$. For any one parameter subgroup $\lambda : \bG_m \to T$, the Bialynicki-Birula stratification of $X$ can be interpreted as a $\Theta$-stratification of $X/\bG_m$. If we choose $\lambda$ to be suitably generic, then the centers of the strata will be the connected components $Z_\alpha$ of $X^T$. If we assume that the strata cover $X$, for instance if $X$ is proper, then \autoref{thm:virtual_localization} implies that
\[
\chi(X/\bG_m,F) = \sum (-1)^i \dim ((H_i R\Gamma(X,F))^{\bG_m}) = \sum_\alpha \chi(Z_\alpha, (F|_{Z_\alpha} \otimes E_\alpha)^{\bG_m}).
\]
This is comparable to the virtual localization formula in cohomology \cite{graber1999localization} via the equivariant Grothendieck-Riemann-Roch theorem.
\end{ex}

\begin{ex}[Equivariant Verlinde formula]
In \cite{halpern2016equivariant} we use \autoref{thm:virtual_localization} and \autoref{thm:derived_Kirwan} to establish a version of the Verlinde formula on the moduli stack of semistable $G$-principal Higgs bundles $\cH iggs_G^{\ss}$ on a fixed smooth curve $\Sigma$ over $\bC$. It states that for certain ``positive" line bundles $\cL$, the cohomology $H^i(\cH iggs_G^{\ss},\cL) = 0$ for $i>0$ and gives an explicit formula for the ``graded dimension"
\[
\dim_{\bC^\ast}(H^0(\mathcal{H}iggs_G^{\ss},\cL)) = \sum_n t^n \dim \left(H^0(\mathcal{H}iggs_G^{\ss},\cL)_{\op{weight } n}\right),
\]
where $H^0(\mathcal{H}iggs_G^{\ss},\cL)_{\op{weight } n}$ denotes the weight $n$ direct summand of $H^0(\mathcal{H}iggs_G^{\ss},\cL)$ with respect to the $\bG_m$ action on $\cH iggs_G^{\ss}$ which scales the Higgs field. The idea is to use previous results of Teleman and Woodward \cite{teleman2009index} to compute the $K$-theoretic graded index $\cL$ on the stack $\cH iggs$ itself, and prove that the cohomology vanishes there, then to use the methods of derived $\Theta$-stratifications discussed above to identify $R\Gamma(\cH iggs_{G}^{\ss},\cL)$ with $R\Gamma(\cH iggs_G,\cL)$.
\end{ex}

\begin{ex}[Wall-crossing]
Another thing that \autoref{thm:virtual_localization} is well suited for is comparing the index of tautological classes on $\cX^{\ss}$ as one varies the notion of stability. For instance let $X$ be a projective $K3$-surface and consider a variation of polarization $H \in NS(X)_\bR$. For any class in $K^0(\cX)$ represented by some $F \in \Perf(\stackCoh(X)_P)$, one can apply \autoref{thm:virtual_localization} to the two stratifications of \autoref{eqn:wall_crossing} to obtain a wall-crossing formula
\begin{multline}
\chi(\stackCoh(X)^{H_- \ndash\ss}_P,F) - \chi(\stackCoh(X)^{H_+ \ndash \ss}_P,F) = \\ \sum_\alpha \left( \chi(\stackCoh(X)^{H_+ \ndash\ss}_\alpha,F \otimes E_\alpha^{H_+}) - \chi(\stackCoh(X)^{H_- \ndash\ss}_\alpha,F \otimes E_\alpha^{H_-}) \right)
\end{multline}
\end{ex}

\bibliography{theta_stratifications}{}
\bibliographystyle{plain}

\end{document}